\documentclass[12pt]{amsart}
\usepackage{mathrsfs}
\usepackage{}
\usepackage{amsmath}
\usepackage{amsfonts}
\usepackage{amssymb}
\usepackage[all,cmtip]{xy}           
\usepackage{shuffle}
\usepackage{caption}
\usepackage{bbding}
\usepackage{txfonts}
\usepackage[shortlabels]{enumitem}
\usepackage{ifpdf}
\usepackage{xspace}
\ifpdf
\usepackage[colorlinks,final,backref=page,hyperindex]{hyperref}
\else
\usepackage[colorlinks,final,backref=page,hyperindex,hypertex]{hyperref}
\fi
\usepackage{tikz}
\usepackage[active]{srcltx}

\makeatletter

\newtheorem{theorem}{Theorem}[section]
\newtheorem{prop}[theorem]{Proposition}
\newtheorem{lemma}[theorem]{Lemma}
\newtheorem{coro}[theorem]{Corollary}
\newtheorem{prop-def}{Proposition-Definition}[section]

\theoremstyle{definition}
\newtheorem{defn}[theorem]{Definition}

\newtheorem{remark}[theorem]{Remark}
\newtheorem{exam}[theorem]{Example}

\newtheorem{appn}[theorem]{Application}

\newcommand{\nc}{\newcommand}

\newcommand {\emptycomment}[1]{}

\nc{\delete}[1]{{}}
\nc{\mmargin}[1]{}

\nc{\mlabel}[1]{\label{#1}}  
\nc{\mcite}[1]{\cite{#1}}  
\nc{\mref}[1]{\ref{#1}}  
\nc{\meqref}[1]{\eqref{#1}}  
\nc{\mbibitem}[1]{\bibitem{#1}} 

\newcommand{\g}{\mathfrak g}

\newcommand{\uu}{\mathscr U}

\newcommand{\huaO}{{\mathcal{OT}}}

\newcommand{\bk}{{\mathbf{k}}}

\nc{\vep}{\varepsilon}
\nc{\bin}[2]{ (_{\stackrel{\scs{#1}}{\scs{#2}}})}  
\nc{\binc}[2]{(\!\! \begin{array}{c} \scs{#1}\\
		\scs{#2} \end{array}\!\!)}  
\nc{\bincc}[2]{  ( {\scs{#1} \atop
		\vspace{-1cm}\scs{#2}} )}  
\nc{\oline}[1]{\overline{#1}}
\nc{\mapm}[1]{\lfloor\!|{#1}|\!\rfloor}
\nc{\bs}{\bar{S}}
\nc{\cast}{{\,\mbox{\raisebox{.8pt}{$\scriptstyle \circledast$}}\,}}
\nc{\la}{\longrightarrow}
\nc{\ot}{\otimes}
\nc{\rar}{\rightarrow}
\nc{\dar}{\downarrow}
\nc{\dap}[1]{\downarrow \rlap{$\scriptstyle{#1}$}}
\nc{\defeq}{\stackrel{\rm def}{=}}
\nc{\dis}[1]{\displaystyle{#1}}
\nc{\dotcup}{\ \displaystyle{\bigcup^\bullet}\ }
\nc{\hcm}{\ \hat{,}\ }
\nc{\hts}{\hat{\otimes}}
\nc{\hcirc}{\hat{\circ}}
\nc{\lleft}{[}
\nc{\lright}{]}
\nc{\curlyl}{\left \{ \begin{array}{c} {} \\ {} \end{array}
	\right .  \!\!\!\!\!\!\!}
\nc{\curlyr}{ \!\!\!\!\!\!\!
	\left . \begin{array}{c} {} \\ {} \end{array}
	\right \} }
\nc{\longmid}{\left | \begin{array}{c} {} \\ {} \end{array}
	\right . \!\!\!\!\!\!\!}
\nc{\ora}[1]{\stackrel{#1}{\rar}}
\nc{\ola}[1]{\stackrel{#1}{\la}}
\nc{\scs}[1]{\scriptstyle{#1}} \nc{\mrm}[1]{{\rm #1}}
\nc{\dirlim}{\displaystyle{\lim_{\longrightarrow}}\,}
\nc{\invlim}{\displaystyle{\lim_{\longleftarrow}}\,}
\nc{\dislim}[1]{\displaystyle{\lim_{#1}}} \nc{\colim}{\mrm{colim}}
\nc{\mvp}{\vspace{0.3cm}} \nc{\tk}{^{(k)}} \nc{\tp}{^\prime}
\nc{\ttp}{^{\prime\prime}} \nc{\svp}{\vspace{2cm}}
\nc{\vp}{\vspace{8cm}}
\nc{\modg}[1]{\!<\!\!{#1}\!\!>}
\nc{\intg}[1]{F_C(#1)}
\nc{\lmodg}{\!<\!\!}
\nc{\rmodg}{\!\!>\!}
\nc{\cpi}{\widehat{\Pi}}
\nc{\labs}{\mid\!}
\nc{\rabs}{\!\mid}
\nc{\btr}{\blacktriangleright}
\nc{\bmot}{\otimes_\lR^b}

\newcommand{\MM}{\ensuremath{\mathcal{M}}\xspace}
\newcommand{\XM}{\mathfrak{X}(\MM)}
\newcommand{\CM}{\mathcal{C}^\infty(\MM)}
\newcommand{\XMF}{\mathfrak{X}_F(\MM)}
\newcommand{\CMF}{\mathcal{C}^\infty_F(\MM)}
\DeclareMathOperator{\Span}{Span}

\nc{\ad}{\mrm{ad}}
\nc{\rRB}{\mathsf{rRB}}
\nc{\cocrRB}{\mathsf{cocrRB}}
\nc{\PH}{\mathsf{PH}}
\nc{\cocPH}{\mathsf{cocPH}}
\nc{\ann}{\mrm{ann}}
\nc{\Ad}{\mrm{Coad}}
\nc{\Aut}{\mrm{Aut}}
\nc{\Der}{\mrm{Der}}
\nc{\Sym}{\mrm{Sym}}
\nc{\br}{\mrm{bre}}
\nc{\can}{\mrm{can}}
\nc{\Cont}{\mrm{Cont}}
\nc{\rchar}{\mrm{char}}
\nc{\cok}{\mrm{coker}}
\nc{\de}{\mrm{dep}}
\nc{\dtf}{{R-{\rm tf}}}
\nc{\dtor}{{R-{\rm tor}}}

\nc{\Dif}{\mrm{Diff}}
\nc{\Div}{\mrm{Div}}
\nc{\End}{\mrm{End}}
\nc{\Ext}{\mrm{Ext}}
\nc{\Fil}{\mrm{Fil}}
\nc{\Fr}{\mrm{Fr}}
\nc{\Frob}{\mrm{Frob}}
\nc{\Gal}{\mrm{Gal}}
\nc{\GL}{\mrm{GL}}
\nc{\Gr}{\mrm{Gr}}
\nc{\Hom}{\mrm{Hom}}
\nc{\Hoch}{\mrm{Hoch}}
\nc{\hsr}{\mrm{H}}
\nc{\hpol}{\mrm{HP}}
\nc{\id}{\mrm{id}}
\nc{\im}{\mrm{im}}
\nc{\inv}{\mrm{inv}}
\nc{\Id}{\mrm{Id}}
\nc{\ID}{\mrm{ID}}
\nc{\Irr}{\mrm{Irr}}
\nc{\incl}{\mrm{incl}}
\nc{\length}{\mrm{length}}
\nc{\NLSW}{\mrm{NLSW}}
\nc{\Lie}{\mrm{Lie}}
\nc{\mchar}{\rm char}
\nc{\mpart}{\mrm{part}}
\nc{\ql}{{\QQ_\ell}}
\nc{\qp}{{\QQ_p}}
\nc{\rank}{\mrm{rank}}
\nc{\rcot}{\mrm{cot}}
\nc{\rdef}{\mrm{def}}
\nc{\rdiv}{{\rm div}}
\nc{\rtf}{{\rm tf}}
\nc{\rtor}{{\rm tor}}
\nc{\res}{\mrm{res}}
\nc{\SL}{\mrm{SL}}
\nc{\Spec}{\mrm{Spec}}
\nc{\tor}{\mrm{tor}}
\nc{\Tr}{\mrm{Tr}}
\nc{\tr}{\mrm{tr}}
\nc{\wt}{\mrm{wt}}

\nc{\bfk}{{\bf k}}
\nc{\bfone}{{\bf 1}}
\nc{\bfzero}{{\bf 0}}
\nc{\detail}{\marginpar{\bf More detail}
	\noindent{\bf Need more detail!}
	\svp}
\nc{\gap}{\marginpar{\bf Incomplete}\noindent{\bf Incomplete!!}
	\svp}
\nc{\FMod}{\mathbf{FMod}}
\nc{\Int}{\mathbf{Int}}
\nc{\Mon}{\mathbf{Mon}}
\nc{\remarks}{\noindent{\bf Remarks: }}
\nc{\Rep}{\mathbf{Rep}}
\nc{\Rings}{\mathbf{Rings}}
\nc{\Sets}{\mathbf{Sets}}
\nc{\Diff}{\mathbf{Diff}}
\nc{\Inte}{\mathbf{Inte}}
\nc{\U}{\mathrm{U}}

\newcommand{\CWM}{C^{\infty}(M)}

\nc{\BB}{{\mathbb B}}   \nc{\CC}{{\mathbb C}}
\nc{\DD}{{\mathbb D}}   \nc{\EE}{{\mathbb E}}
\nc{\FF}{{\mathbb F}}   \nc{\GG}{{\mathbb G}}
\nc{\HH}{{\mathbb H}}   \nc{\LL}{{\mathbb L}}
\nc{\NN}{{\mathbb N}}   \nc{\PP}{{\mathbb P}}
\nc{\QQ}{{\mathbb Q}}   \nc{\RR}{{\mathbb R}}
\nc{\TT}{{\mathbb T}}   \nc{\VV}{{\mathbb V}}
\nc{\ZZ}{{\mathbb Z}}   \nc{\TP}{\widetilde{P}}

\nc{\cala}{{\mathcal A}}    \nc{\calc}{{\mathcal C}}
\nc{\cald}{\mathcal{D}}     \nc{\cale}{{\mathcal E}}
\nc{\calf}{{\mathcal F}}    \nc{\calg}{{\mathcal G}}
\nc{\calh}{{\mathcal H}}    \nc{\cali}{{\mathcal I}}
\nc{\call}{{\mathcal L}}    \nc{\calm}{{\mathcal M}}
\nc{\caln}{{\mathcal N}}    \nc{\calo}{{\mathcal O}}
\nc{\calp}{{\mathcal P}}    \nc{\calr}{{\mathcal R}}

\nc{\cals}{{\mathcal S}}    \nc{\calt}{{\Omega}}
\nc{\calv}{{\mathcal V}}    \nc{\calw}{{\mathcal W}}
\nc{\calx}{{\mathcal X}}

\nc{\fraka}{{\mathfrak a}}
\nc{\frakb}{\mathfrak{b}}
\nc{\frakg}{\mathfrak{g}}
\nc{\frakk}{{\frak k}}
\nc{\frakl}{{\frak l}}
\nc{\fraks}{{\frak s}}
\nc{\frakB}{{\frak B}}
\nc{\frakm}{{\frak m}}
\nc{\frakM}{{\frak M}}
\nc{\frakp}{{\frak p}}
\nc{\frakW}{{\frak W}}
\nc{\frakX}{{\frak X}}
\nc{\frakS}{{\frak S}}
\nc{\frakA}{{\frak A}}
\nc{\frakx}{{\frakx}}

\nc{\ynr}[1]{\textcolor{orange}{\underline{Yunnan:}#1 }}

\nc{\lir}[1]{\textcolor{red}{\underline{Li:}#1 }}

\nc{\adir}[1]{\textcolor{blue}{\underline{Adrien:}#1 }}

\nc{\lR}{R}
\nc{\lL}{L}

\begin{document}

\title[]{Post-Hopf algebroids, post-Lie-Rinehart algebras and geometric numerical integration}

\author{Adrien Busnot Laurent}
\address{Univ Rennes, Inria (Research team MINGuS), IRMAR (CNRS UMR 6625) and ENS Rennes, France}
\email{Adrien.Busnot-Laurent@inria.fr}

\author{Yunnan Li}
\address{School of Mathematics and Information Science, Guangzhou University,
Guangzhou 510006, China}
\email{ynli@gzhu.edu.cn}

\author{Yunhe Sheng}
\address{Department of Mathematics, Jilin University, Changchun 130012, Jilin, China}
\email{shengyh@jlu.edu.cn}

\begin{abstract}
In this paper, we introduce the notion of post-Hopf algebroids, generalizing the pre-Hopf algebroids introduced in \cite{BL} in the study of exotic aromatic S-series. We construct action post-Hopf algebroids through actions of post-Hopf algebras. We show that the universal enveloping algebra of a post-Lie-Rinehart algebra (post-Lie algebroid) is  naturally a post-Hopf algebroid. As a byproduct, we construct the free post-Lie-Rinehart algebra using  a magma algebra with a linear map to the derivation Lie algebra of a commutative associative algebra.  Applications in geometric numerical integration on manifolds are given. \end{abstract}


\renewcommand{\thefootnote}{}
\footnotetext{2020 Mathematics Subject Classification.
17B38, 
16T05, 
17B35, 
41A58, 
65L06. 
}
\keywords{post-Hopf algebroid,  post-Lie algebroid, post-Lie-Rinehart algebra, numerical   integration}

\maketitle
%

\tableofcontents

\allowdisplaybreaks
\section{Introduction}

The first purpose of this paper is to study the universal enveloping of a post-Lie algebroid by introducing the notion of post-Hopf algebroids. The second purpose is to give applications of post-Hopf algebroids in braiding operators on Hopf algebroids and numerical analysis on manifolds.

\subsection{Post-Lie algebras, post-Hopf algebras and post-Lie algebroids}

The notion of a post-Lie algebra was introduced by Vallette from his study of Koszul duality of operads in \mcite{Val}, and can be applied to the study of generalized Lax pairs \mcite{BGN} and regularity structures \cite{BK}.
Munthe-Kaas and Lundervold found that post-Lie algebras also naturally appear in differential geometry and play important roles in the study of numerical integration on  manifolds \mcite{ML}.
As the integration of post-Lie algebras, the notion of post-Lie groups was introduced in \cite{PostG}, which can be used to construct set-theoretical solutions of the Yang-Baxter equation. See \cite{AEM} for good examples of post-groups.

In \cite{LST}, the notion of post-Hopf algebras was introduced as an abstraction from two basic examples, namely the group algebra of a post-group and the universal enveloping algebra of a post-Lie algebra. A novel property is that any cocommutative post-Hopf algebra gives rise to a sub-adjacent Hopf algebra with a generalized Grossman-Larson product. There is naturally a post-Lie algebra structure on the space of primitive elements, and a post-group structure on the set of group-like elements in an arbitrary post-Hopf algebra. Also, cocommutative post-Hopf algebras are essentially equivalent to cocommutative Hopf braces, and thus naturally provide Yang-Baxter operators~\cite{LST,ZZMG}. See~\cite{VRP,Sci} for two different kinds of generalizations of post-Hopf algebras along this connection.

 The notion of Lie algebroids was introduced by Pradines in 1967, which are generalizations of
 Lie algebras and tangent bundles. Just as Lie algebras are the infinitesimal objects of Lie groups,
 Lie algebroids are the infinitesimal objects of Lie groupoids. Lie algebroids and Lie groupoids play important roles in differential geometry, foliation theory, Poisson geometry and mathematical physics.  See \cite{Ma} for general theories about Lie
 algebroids and Lie groupoids.
As the geometric generalization of a post-Lie algebra, the notion of a post-Lie algebroid was also introduced by Munthe-Kaas and Lundervold in their study of numerical integration~\cite{ML}. See also \cite{MSV} for further applications.

\subsection{Hopf algebroids and pre-Hopf algebroids}

Roughly speaking, bialgebroids, resp. Hopf algebroids, are the appropriate generalization of bialgebras, resp. Hopf algebras over a commutative ring (often a field) to algebraic structures over a noncommutative base algebra.
Bialgebroids appeared independently at different times and places, in algebra, algebraic topology and Poisson geometry, in various versions of generality. The notion of a bialgebroid in its full generality was given by Takeuchi in 1970s, under the name $\times_R$-bialgebras. Later his ideas were rediscovered independently by J.-H. Lu \cite{Lu}.
Similarly to the relation between Hopf algebras and groups, Hopf algebroids can be viewed as a quantization of groupoids.
The term of Hopf algebroids was first introduced by Haynes R. Miller in his 1975 doctoral dissertation for the purpose to describe the $E_2$-term of the Adams-Novikov spectral sequence, and also with other applications in algebraic topology.
In \cite{Mal}, Maltsiniotis gave a definition of Hopf algebroids with commutative base algebras and their images under the source and target maps lie in the centers of total algebras.
In \cite{Lu}, Lu generalized the notion of Hopf algebroids, in which neither the total algebras nor the base algebras are required to be commutative.
After that there appear several other different definitions of Hopf algebroids for various purposes; see \cite{BS,Sch,HM}.

In this paper, we adopt Lu's definition of Hopf algebroids in \cite{Lu}, fit for our fundamental example: the universal enveloping algebra of a Lie algebroid (Lie-Rinehart algebra).
So far the bialgebroid structure on the universal enveloping algebra of a Lie algebroid has been extensively studied with fruitful results; see e.g. \cite{MM,BKS}. On the other hand, a natural construction of cocommutative Hopf algebroids from a finite groupoid was given
in \cite{Gh}, generalizing the group algebra of a group.

Recently, Bronasco and the first author uncovered in \cite{BL} that the pre-Hopf algebroid and the Grossman-Larson Hopf algebroid are the underlying algebraic structures associated to the laws of composition and substitution of exotic aromatic S-series, and used them to provide the algebraic foundations of stochastic numerical analysis (see also \cite{LV2}).
They later extended these structures with planar rooted trees and naturally obtained post-Hopf algebras in \cite{BBH}.
This allowed for the creation of high-order intrinsic numerical methods for stochastic differential equations on manifolds.

\subsection{Main results and outline of the paper}
The purpose of this paper is twofold. On the one hand, we study the universal enveloping algebra of a post-Lie-Rinehart algebra (post-Lie algebroid). On the other hand, we generalize the pre-Hopf algebroid introduced in  \cite{BL} to the post-algebraic structure context, and give applications in geometric numerical integration.

For these purposes, first we introduce the notion of post-Hopf algebroids, and show that a post-Hopf algebroid gives rise to a Hopf algebroid, which is called the Grossman-Larson Hopf algebroid. Through actions of post-Hopf algebras, we construct a class of examples, called action post-Hopf algebroids. We show that post-Hopf algebroids give rise to braiding operators on the associated Grossman-Larson Hopf algebroid. Then we introduce the notion of post-Lie-Rinehart algebras as the algebraic analog of post-Lie algebroids, and construct the free object using a magma algebra with a linear map to $\Der(R)$, where $R$ is a commutative associative algebra. The universal enveloping algebra of a post-Lie-Rinehart algebra is naturally a post-Hopf algebroid. Applications of this universal enveloping post-Hopf algebroid and the action post-Hopf algebroid in numerical analysis are provided for the study of volume-preserving Lie-group methods.

The article is organized as follows. In Section~\ref{sec:post-hopf}, we  introduce the notion of (weak) post-Hopf algebroids and show the existence of a sub-adjacent Hopf algebroid associated with a post-Hopf algebroid, so-called the Grossman-Larson Hopf algebroid. A useful construction called the action post-Hopf algebroid is provided (Theorem~\ref{thm:action-post-Hopf}). Also, a post-Hopf algebroid naturally induces a braiding operator on the Grossman-Larson Hopf algebroid~(Theorem~\ref{thm:ph-braiding}).
In Section~\ref{sec:post-LR}, first we introduce the notion of post-Lie-Rinehart algebras and give the construction of their free objects (Theorem~\ref{thm:free-post-LR}).
Then we show that the universal enveloping algebra of a post-Lie-Rinehart algebra is a (weak) post-Hopf algebroid (Theorem~\ref{thm:uea-post-LR}). In Section \ref{sec:app}, we outline applications of post-Lie-Rinehart algebras and post-Hopf algebroids in numerical analysis.

\vspace{0.1cm}

\noindent
{\bf Convention.}
In this paper, we fix an algebraically closed ground field $\bk$ of characteristic 0 and a commutative associative unital algebra $\lR$. All the objects under discussion, including vector spaces, linear maps, algebras and tensor products, are taken over $\bk$ by default if without emphasis.

 \vspace{0.1cm}
 \noindent
{\bf Acknowledgements. } This research is supported by NSFC (Grant No. 12471060, 12071094, W2412041) and ANR-25-CE40-2862-01 (MaStoC - Manifolds and Stochastic Computations).
We give warmest thanks to Xiao Han for helpful comments.

\section{Post-Hopf algebroids}\label{sec:post-hopf}

In this section, we introduce the notion of post-Hopf algebroids, which are generalizations of post-Hopf algebras. In particular, we give the action post-Hopf algebroids using actions of post-Hopf algebras. We introduce the notion of braiding operators on Hopf algebroids, and show that post-Hopf algebroids give rise to braiding operators on the Grossman-Larson Hopf algebroids.

\subsection{Bialgebroids and Hopf algebroids}

First we adopt the following definition of Hopf algebroids with commutative base algebras as a special situation in \cite{Lu}; see also \cite[\S~2.1]{KT}.
\begin{defn}\label{defi:Hopf-algebroid}
Take the commutative unital algebra $\lR$ as the base algebra, and $H$ another unital algebra as the total algebra.
  A {\bf bialgebroid} structure on $H$ over $\lR$ consists of the following data:
\begin{enumerate}[(i)]
\item
An algebra homomorphism $\iota:\lR\to H$ as the source and also target map, giving the following canonical left $\lR$-module structure on $H$,
\begin{align*}
&\lambda:\lR\otimes H\to H,\quad f\otimes x\mapsto \iota(f)x,
\end{align*}
and we write $\iota(f)x$ as $fx$ for short.
\item
An $\lR$-module map $\Delta:H\to H\otimes_\lR H,\ x\to x_1\otimes x_2$ as the coproduct, satisfying
\begin{enumerate}[(a)]
\item
$(\Delta\otimes_\lR\id_H)\Delta = (\id_H\otimes_\lR\Delta)\Delta$.
\item
$\Delta(x)(\iota(f)\otimes 1_H-1_H\otimes \iota(f))=0$ for any $f\in \lR$ and $x\in H$,
so inducing an algebra structure on the image $\Delta(H)$.
\item
$\Delta$ is an algebra homomorphism from $H$ to $\Delta(H)$. In particular, $\Delta(1_H)=1_H \otimes_\lR 1_H$.
\end{enumerate}
\item
An $\lR$-module map $\varepsilon:H\to \lR$ as the counit, satisfying
\begin{enumerate}[(a)]
\item
$\varepsilon(1_H)=1_\lR$.
\item
$\ker\varepsilon$ is a left ideal of $H$, namely $\varepsilon(x(y-\iota\varepsilon(y)))=0$ for any $x,y\in H$.
\item
$\lambda(\varepsilon\otimes\id_H)\Delta=\id_H$ and  $\lambda(\varepsilon\otimes\id_H)\Delta^{\rm op}=\id_H$, where $\Delta^{\rm op}$ is the opposite of $\Delta$.
\end{enumerate}
\end{enumerate}

\quad We say  that $H$ is cocommutative if $\Delta=\Delta^{\rm op}$.

  A {\bf Hopf algebroid} is a bialgebroid $H$ over $\lR$ which admits an algebra anti-isomorphism $S:H\to H$ as the antipode map, satisfying
\begin{enumerate}[(a)]
\item
$S\iota=\iota$,
\item
$m_H(S\otimes\id_H)\Delta = \iota\varepsilon S$,
\item
$m_H(\id_H\otimes S)\Delta = \iota\varepsilon$.
\end{enumerate}
\end{defn}

\begin{remark}
It is straightforward to see that if $\iota(\lR)$ lies in the center $Z(H)$ of $H$, then $H\otimes _\lR H$ is an $\lR$-algebra, and both the coproduct $\Delta$ and the counit $\varepsilon$ are $R$-algebra homomorphisms. In this situation, we call $H$ an {\bf $\lR$-bialgebra}.
\end{remark}

\begin{defn}\label{defn:Hopf-R-algebra}
A Hopf algebroid $H$ over $\lR$ is called a {\bf Hopf $\lR$-algebra}, if $\iota(\lR)$ lies in the center $Z(H)$ of $H$, the antipode map $S$ of $H$ is $\lR$-linear and $\varepsilon S=\varepsilon$.
\end{defn}
Note that all structural maps of a Hopf $\lR$-algebra are $\lR$-linear.

\begin{exam}\label{ex:tensor-Hopf}
Let $(A,\Delta_A,\varepsilon_A,S_A)$ be a Hopf algebra. The tensor algebra $H=\lR\otimes A$ has a natural Hopf $\lR$-algebra structure $(H,\Delta_H,\varepsilon_H,S_H,\iota)$, where
\begin{eqnarray*}
\iota(f) &=& f1_A,\\
\Delta_H(fa) &=& fa_1\otimes_\lR 1_\lR a_2,\\
\varepsilon_H(fa) &=& f\varepsilon_A(a),\\
S_H(fa) &=& fS_A(a)
\end{eqnarray*}
for any $fa\in \lR\otimes A$ with $f\in \lR$ and $a\in A$. We call $H=\lR\otimes A$ the  {\bf tensor product Hopf $\lR$-algebra}.
\end{exam}

\begin{lemma}\label{lem:Hopf-R-algebra}
If $H$ is a   cocommutative Hopf algebroid over $\lR$ such that $\iota(\lR)\subseteq Z(H)$, then $H$ is a Hopf $\lR$-algebra.
\end{lemma}
\begin{proof}
We only need to show that $S$   is $\lR$-linear and $\varepsilon S=\varepsilon$. First for any $f\in R$, $x\in H$, we have
\begin{eqnarray*}
 S(fx) &=& S(x)S(\iota(f))\ = \ S(x)\iota(f)\ =\ fS(x),
\end{eqnarray*}
which implies that $S$ is $R$-linear. Also, since the counit $\varepsilon$ becomes an algebra homomorphism, we have
\begin{eqnarray*}
\varepsilon(S(x)) &=& \varepsilon(\iota(\varepsilon(S(x))))=\varepsilon(S(x_1)x_2)=\varepsilon(S(x_1))\varepsilon(x_2)\\
&=& \varepsilon(x_1)\varepsilon(S(x_2))= \varepsilon(x_1S(x_2))= \varepsilon(\iota(\varepsilon(x)))\\
& =&\varepsilon(x),
\end{eqnarray*}
which implies that $\varepsilon S=\varepsilon$.  Hence, $H$ is a Hopf $\lR$-algebra by definition.
\end{proof}

\begin{defn}\label{defn:mod-alg}
Given a bialgebroid $H$ over $\lR$, an $\lR$-algebra $A$ is called a {\bf module algebra} over $H$ via an  action $\rightharpoonup:H\otimes A\to A$, if
\begin{eqnarray}
xy\rightharpoonup a&=&x\rightharpoonup (y\rightharpoonup a),\\
1_H\rightharpoonup a&=&a,\\
\label{eq:mod-alg-0}
fx\rightharpoonup a &=& f(x\rightharpoonup a),\\
\label{eq:mod-alg-1}
x\rightharpoonup ab &=& (x_1\rightharpoonup a)(x_2\rightharpoonup b),\\
\label{eq:mod-alg-2}
x\rightharpoonup 1_A &=& \varepsilon(x)1_A
\end{eqnarray}
for any $f\in\lR$, $x\in H$ and $a,b\in A$. When $\lR=\bk$, it recovers the notion of a module algebra over a bialgebra.
\end{defn}

 For later use, we need the following
construction of Hopf algebroids from smash product algebras.  Let $(A,\Delta_A,\varepsilon_A,S_A)$ be a cocommutative Hopf algebra such that $\lR$ is an $A$-module algebra via an action $\rightharpoonup$. Then there is an algebra structure on $\lR\otimes A$ given by
$$
fa\cdot gb = f(a_1\rightharpoonup g)a_2b,
$$
for all $fa,gb\in\lR\otimes A$. This algebra is called the {\bf smash product algebra} and denoted by $\lR\# A$.

\begin{theorem}\label{prop:smash-prod}
With the above notations,  the smash product algebra $H=\lR\# A$ has a natural Hopf algebroid structure $(H,\Delta_H,\varepsilon_H,S_H,\iota)$ over $\lR$, where the   comultiplication $\Delta_H$, the counit $\varepsilon_H$ and the antipode $S_H$ are defined as follows:
\begin{eqnarray*}
\Delta_H(fa) &=& fa_1\otimes_\lR 1_\lR a_2,\\
\varepsilon_H(fa) &=& f\varepsilon_A(a),\\
S_H(fa) &=& (S_A(a_1)\rightharpoonup f)S_A(a_2)
\end{eqnarray*}
for any $f,g\in \lR$ and $a,b\in A$, and $\iota:\lR\to H$ is the canonical inclusion of algebras. We call $H$ an {\bf action Hopf algebroid} over $\lR$.
\end{theorem}
\begin{proof}
 According to \cite[Theorem~4.1]{BM}, given a braided commutative algebra $A$ in the Yetter-Drinfeld module category ${_H}\mathscr Y\!\mathscr D^H$ over a Hopf algebra $H$,
the smash product algebra $A\# H$ naturally possesses a Hopf algebroid structure over $A$ in the sense of Lu. Such a result is a generalization of \cite[Theorem~5.1]{Lu}.

Now under our notations, we take the commutative algebra $\lR$ as a module algebra over a given cocommutative Hopf algebra $A$ and choose the trivial coaction of $A$ on $\lR$, then $\lR$ is clearly a braided commutative algebra in ${_A}\mathscr Y\!\mathscr D^A$, and so \cite[Theorem~4.1]{BM} can be applied to give the desired statement. Note that in this special situation, both the source map and the target map are the canonical inclusion, and we just denote it by $\iota$ as previously.
\end{proof}

\subsection{Post-Hopf algebroids}

In order to introduce the key notion of post-Hopf algebroids, we first recall post-Hopf algebras introduced in \cite{LST}. The group algebra of a post-group and the universal enveloping algebra of a post-Lie algebra are two fundamental examples.
\begin{defn}[{\cite[Definition~2.1]{LST}}]\label{defi:pH}
A {\bf post-Hopf algebra} is a pair $(H,\rhd)$, where $H$ is a Hopf algebra and $\rhd:H\otimes H\to H$ is a coalgebra homomorphism satisfying the following equalities:
\begin{eqnarray}
\label{Post-2}x\rhd (y\cdot z)&=&(x_1\rhd y)\cdot(x_2\rhd z),\\
\label{Post-4}x\rhd (y\rhd z)&=&\big(x_1\cdot(x_2\rhd y)\big)\rhd z,
\end{eqnarray}
for any $x,y,z\in H$, and the left multiplication $\alpha_\rhd:H\to \End(H)$ defined by
$\alpha_{\rhd, x} y= x\rhd y$ for all $x,y\in H,$
is convolution invertible in $\Hom(H,\End(H))$. Namely, there exists unique $\beta_\rhd:H\to\End(H)$ such that
\begin{equation}\label{Post-Hopf-con}
\alpha_{\rhd,x_1}\circ\beta_{\rhd,x_2}=\beta_{\rhd,x_1}\circ\alpha_{\rhd,x_2}=\varepsilon(x)\id_H,\quad\forall x\in H.
\end{equation}

\end{defn}

Moreover, we have the following properties.

\begin{lemma}[{\cite[Lemma~2.4]{LST}}]
Let $(H,\rhd)$ be a post-Hopf algebra. For any $x,y\in H$, we have
\begin{eqnarray}
\label{Post-1}x\rhd 1_H&=&\varepsilon(x)1_H,\\
\label{Post-3}1_H\rhd x&=&x,\\
\label{Post-5}S(x\rhd y)&=&x\rhd S(y).
\end{eqnarray}
\end{lemma}

\begin{theorem}[{\cite[Theorem~2.5]{LST}}]\label{thm:subHopf}
Let $(H,\rhd)$ be a cocommutative post-Hopf algebra. Then
$$H_\rhd\coloneqq(H,*_\rhd,1_H,\Delta,\varepsilon,S_\rhd)$$
is a Hopf algebra, which is called the {\bf sub-adjacent Hopf algebra}, where for all $x,y\in H$,
\begin{eqnarray}
\label{post-rbb-1}x *_\rhd y&\coloneqq&x_1\cdot (x_2\rhd y),\\
\label{post-rbb-2}S_\rhd(x)&\coloneqq&\beta_{\rhd,x_1}(S(x_2)),
\end{eqnarray}
and $*_\rhd$ is called the {\bf Grossman-Larson product}.

Furthermore, $H$ is a left $H_\rhd$-module bialgebra via the action $\rhd$, so $\beta_{\rhd,x}=\alpha_{\rhd,S_\rhd(x)}$.

\end{theorem}

Now we are in position to define post-Hopf algebroids generalizing cocommutative post-Hopf algebras. First we introduce a weaken version, whose counterpart in the original Hopf algebra case was given
by getting rid of the  restrictive  condition that the left multiplication $\alpha_\rhd$ is convolution invertible.
\begin{defn}\label{defi:weak-post-Hopf}
A {\bf weak post-Hopf algebroid} over $\lR$
consists of a {\it cocommutative} Hopf $\lR$-algebra $(H,\iota)$,
and  a $\bk$-linear map $\rhd:H\otimes H\to H$ such that
\begin{eqnarray}
\label{eq:post-HAD-1}
\Delta(x\rhd y) &=& (x_1\rhd y_1)\otimes_\lR (x_2\rhd y_2),\\
\label{eq:post-HAD-1'}
x\rhd 1_H &=& \iota(\varepsilon(x)),\\
\label{eq:post-HAD-1''}
1_H\rhd x &=& x,\\
\label{eq:post-HAD-2'}
\iota(\varepsilon(x\rhd y)) &=& x\rhd \iota(\varepsilon(y)),\\
\label{eq:post-HAD-2}
fx\rhd y &=& f(x\rhd y),\\
\label{eq:post-HAD-3}
x\rhd yz &=& (x_1\rhd y)(x_2\rhd z),\\
\label{eq:post-HAD-4}
x\rhd(y\rhd z) &=& x_1(x_2\rhd y)\rhd z
\end{eqnarray}
for any $f\in\lR$ and $x,y,z\in H$.
\end{defn}

\begin{lemma}
Let $(H,\iota,\rhd)$ be a weak post-Hopf algebroid over $\lR$. For any $f\in\lR$ and $x\in H$, we have
\begin{eqnarray}
\label{eq:post-HAD-5'}
x\rhd \iota(f) &=& \iota(\varepsilon(x\rhd \iota(f))),\\
\label{eq:post-HAD-5}
\iota(f) \rhd x &=& fx.
\end{eqnarray}
\end{lemma}
\begin{proof}
First we see that
\begin{eqnarray*}
x\rhd \iota(f) &=& x\rhd \iota(\varepsilon(\iota(f))) \ \stackrel{\eqref{eq:post-HAD-2'}}{=}\ \iota(\varepsilon(x\rhd \iota(f))).
\end{eqnarray*}
Namely, the image of $\iota$ is stable under the left multiplication of $\rhd$.
Also, we have
\begin{eqnarray*}
\iota(f) \rhd x
&\stackrel{\eqref{eq:post-HAD-2}}{=}& f(1_H\rhd x)\ \stackrel{\eqref{eq:post-HAD-1''}}{=}\ fx,
\end{eqnarray*}
which finishes the proof.
\end{proof}

\begin{theorem}\label{thm:GL-bialgebroid}
Given a weak post-Hopf algebroid $(H,\iota,\rhd)$ over $\lR$, there exists an associative product $*_\rhd$, called the {\bf Grossman-Larson product}, on $H$ defined by
\begin{eqnarray}\label{eq:GL-product}
x*_\rhd y &\coloneqq& x_1(x_2\rhd y),\quad\forall x,y\in H,
\end{eqnarray}
 such that the tuple $H_\rhd=(H,*_\rhd,\Delta,\varepsilon,\iota)$ forms a bialgebroid over $R$, which is called the {\bf Grossman-Larson bialgebroid} over $\lR$.

  Moreover, $H$ is a module algebra over $H_\rhd$ with respect to the product $\rhd$  in the sense of Definition \ref{defn:mod-alg}.
\end{theorem}
\begin{proof}
For any $f\in R$ and $x,y\in H$, since $\Delta$ is $\lR$-linear and $\iota(\lR)\subseteq Z(H)$, we have
\begin{eqnarray*}
fx_1(x_2\rhd y)=x_1(f(x_2\rhd y))
\stackrel{\eqref{eq:post-HAD-2}}{=} x_1(fx_2\rhd y),
\end{eqnarray*}
which implies that the product $*_\rhd$ is well-defined.

By Eqs.~\eqref{eq:post-HAD-1'} and \eqref{eq:post-HAD-1''}, it is clear that $1_H*_\rhd x =x*_\rhd 1_H=x$.
Also, Eq.~\eqref{eq:post-HAD-1} and the cocommutativity of $H$ implies that
\begin{equation}\label{eq:111}
\Delta(x*_\rhd y)=(x_1*_\rhd y_1)\otimes_\lR (x_2*_\rhd y_2).
\end{equation}
Hence, we have
\begin{eqnarray*}
(x*_\rhd y)*_\rhd z &\stackrel{\eqref{eq:GL-product}}{=}& (x_1*_\rhd y_1)((x_2*_\rhd y_2)\rhd z)\\
&\stackrel{\eqref{eq:post-HAD-4},\eqref{eq:GL-product}}{=}&
x_1(x_2\rhd y_1)(x_3\rhd (y_2\rhd z))\\
&\stackrel{\eqref{eq:post-HAD-3}}{=}&
x_1(x_2\rhd y_1(y_2\rhd z))\\
&\stackrel{\eqref{eq:GL-product}}{=}&  x*_\rhd(y*_\rhd z).
\end{eqnarray*}
Thus, $(H, *_\rhd)$ is an associative algebra also with the unit $1_H$.

By Eq.~\eqref{eq:post-HAD-1''},  $\iota$ is an algebra homomorphism from $R$ to the associative algebra $(H,*_\rhd)$:
$$
\iota(fg)=\iota(f)\iota(g)=\iota(f)*_\rhd\iota(g),\quad \forall f,g\in \lR.
$$
Thus $H_\rhd$ has the same $\lR$-module structure as $H$:
$$\iota(f)*_\rhd x =\iota(f)(1_H\rhd x) \stackrel{\eqref{eq:post-HAD-1''}}{=} \iota(f)x = fx.$$

Also, we have
\begin{eqnarray*}
  (x_1*_\rhd\iota(f))\otimes_\lR x_2 - x_1\otimes_\lR (x_2*_\rhd\iota(f))
&\stackrel{\eqref{eq:GL-product}}{=}& x_1(x_2\rhd\iota(f))\otimes_\lR x_3 - x_1\otimes_\lR x_2(x_3\rhd\iota(f))\\
&=& (x_1\rhd\iota(f))x_2\otimes_\lR x_3 - x_1\otimes_\lR (x_2\rhd\iota(f))x_3\\
&=& 0.
\end{eqnarray*}
Combined with \eqref{eq:111}, we see that (ii) in Definition \ref{defi:Hopf-algebroid} holds for the tuple $(H,*_\rhd,\Delta,\iota)$.

Finally, we have $\iota(\varepsilon(x_1))*_\rhd x_2 \stackrel{\eqref{eq:GL-product}}{=} \iota(\varepsilon(x_1))(1\rhd x_2) \stackrel{\eqref{eq:post-HAD-1''}}{=} \iota(\varepsilon(x_1))x_2=x$, and
\begin{eqnarray*}
\varepsilon(x*_\rhd \iota(\varepsilon(y)))
&\stackrel{\eqref{eq:GL-product}}{=}&
 \varepsilon(x_1(x_2\rhd\iota(\varepsilon(y))))
\stackrel{\eqref{eq:post-HAD-2'}}{=} \varepsilon(x_1\iota(\varepsilon(x_2\rhd y)))
= \varepsilon(x_1)\varepsilon(\iota(\varepsilon(x_2\rhd y)))\\
&=& \varepsilon(x_1)\varepsilon(x_2\rhd y) = \varepsilon(x_1(x_2\rhd y))
\stackrel{\eqref{eq:GL-product}}{=}
\varepsilon(x*_\rhd y),
\end{eqnarray*}
which imply that (iii) in Definition \ref{defi:Hopf-algebroid} holds.
Hence, we have confirmed the desired bialgebroid structure $H_\rhd$ over $\lR$.

It is straightforward to see that $H$ is clearly a module algebra over $H_\rhd$ with respect to the product $\rhd$ by Definition~\ref{defn:mod-alg}.
\end{proof}

\begin{coro}\label{coro:left-action}
For a weak post-Hopf algebroid $(H,\iota,\rhd)$ over $\lR$, there is a left $H_\rhd$-module algebra action on $\lR$ given by:
\begin{eqnarray}\label{eq:left-action}
x(f) &\coloneqq& \varepsilon(x\rhd \iota(f)),\quad \forall f\in \lR,\,x\in H.
\end{eqnarray}
\end{coro}

\begin{proof} By Theorem~\ref{thm:GL-bialgebroid} and Eq.~\eqref{eq:post-HAD-5'}, $\iota(R)$ is a module subalgebra of $H$ over $H_\rhd$. Since
    $\varepsilon\iota=\id_\lR$, there is an induced module algebra structure on $R$ given as desired.
\end{proof}

\begin{remark} The $H_\rhd$-module algebra action on $\lR$ given in Corollary~\ref{coro:left-action} is consistent with the one given by Lu
in  \cite[Definition~2.3]{Lu}  for any bialgebroid over $R$. Thus, for the Grossman-Larson bialgebroid $H_\rhd$, the $H_\rhd$-module algebra action on $H$ given in Theorem~\ref{thm:GL-bialgebroid} can be viewed as an expansion of Lu's result.

\end{remark}

\begin{defn}\label{defi:post-Hopf}
A weak post-Hopf algebroid $(H,\iota,\rhd)$ over $\lR$ is called a {\bf post-Hopf algebroid}, if there exists an algebra anti-automorphism $\theta$ of the Grossman-Larson bialgebroid $H_\rhd$
such that
\begin{eqnarray}
\label{post-anti-coalg}
\Delta(\theta(x)) &=& \theta(x_1)\otimes_\lR\theta(x_2),\\
\label{Post-con}
x_1*_\rhd \theta(x_2)&=&\iota(\varepsilon(x)),\\
\label{Post-con'}
\theta(x_1)*_\rhd x_2&=&\iota(\varepsilon(\theta(x)))
\end{eqnarray}
for any $f\in\lR$ and $x\in H$.
\end{defn}
\begin{remark}
The uniqueness of such linear map $\theta$ is an interesting problem. In particular, when $\lR=\bk$, a post-Hopf algebroid over $\lR$ reduces to a cocommutative post-Hopf algebra, and in this case $\theta=S_\rhd$ is unique, where $S_\rhd$ is given by \eqref{post-rbb-2}.
\end{remark}

\begin{lemma}
Let $(H,\iota,\rhd,\theta)$ be a post-Hopf algebroid over $\lR$. For any $f\in\lR$ and $x, y\in H$, we have
\begin{eqnarray}
\label{post-anti-coef}
\theta(fx) &=& (\theta(x_1)\rhd \iota(f))\theta(x_2),\\
\label{anti-theta}
S(x) &=& x_1\rhd \theta(x_2),\\
\label{anti-theta'}
\theta(x) &=& \theta(x_1)\rhd S(x_2),\\
\label{anti-theta''}
\iota(\varepsilon(\theta(x))) &=& \theta(x_1)\rhd \iota(\varepsilon(x_2)),\\
\label{anti-theta''''}
xy &=& x_1*_\rhd (\theta(x_2)\rhd y),\\
\label{anti-theta'''}
x &=& x_1*_\rhd \iota(\varepsilon(\theta(x_2))).
\end{eqnarray}

\end{lemma}
\begin{proof}
First according to Eq.~\eqref{Post-con}, we have
$$
\theta(\iota(f)) = 1_H*_\rhd\theta(\iota(f))  = \iota(f)*_\rhd\theta(1_H)  =
\iota(\varepsilon(\iota(f))) = \iota(f),
$$
as $\Delta(\iota(f))=\iota(f)\otimes_\lR 1_H=1_H\otimes_\lR\iota(f)$.
Then
\begin{eqnarray*}
\theta(fx)= \theta(\iota(f)*_\rhd x)=\theta(x) *_\rhd \theta(\iota(f))
\stackrel{\eqref{post-anti-coalg}, \eqref{eq:GL-product}}{=} \theta(x_1)(\theta(x_2)\rhd \iota(f))= (\theta(x_1)\rhd \iota(f))\theta(x_2).
\end{eqnarray*}
Next since $H$ is a Hopf $\lR$-algebra,
we know that the antipode $S$ is $\lR$-linear and $\varepsilon S=\varepsilon$, so
\begin{eqnarray*}
S(x) &=& S(\varepsilon(x_1)x_2)= S(x_1)\iota(\varepsilon(x_2))\stackrel{\eqref{Post-con}}{=} S(x_1)(x_2*_\rhd \theta(x_3))\stackrel{\eqref{eq:GL-product}}{=} S(x_1)x_2(x_3\rhd \theta(x_4))\\
&=& \iota(\varepsilon(S(x_1)))(x_2\rhd \theta(x_3))\iota(\varepsilon(x_1))(x_2\rhd \theta(x_3)) \stackrel{\eqref{eq:post-HAD-2}}{=}  x_1\rhd \theta(x_2),
\end{eqnarray*}
and then
\begin{eqnarray*}
\theta(x_1)\rhd S(x_2) &\stackrel{\eqref{anti-theta}}{=}& \theta(x_1)\rhd (x_2\rhd \theta(x_3))\stackrel{\eqref{eq:post-HAD-4},\eqref{eq:GL-product}}{=}
(\theta(x_1)*_\rhd x_2)\rhd \theta(x_3)\\
&\stackrel{\eqref{Post-con'}}{=}&
\iota(\varepsilon(\theta(x_1)))\rhd \theta(x_2)\stackrel{\eqref{eq:post-HAD-5}}{=}
\varepsilon(\theta(x_1))\theta(x_2)\ \stackrel{\eqref{post-anti-coalg}}{=}\
\theta(x).
\end{eqnarray*}
Also, we have
\begin{eqnarray*}
\theta(x_1)\rhd \iota(\varepsilon(x_2))
&=&\theta(x_1)\rhd \iota(\varepsilon(S(x_2)))
= \theta(x_1)\rhd S(x_3)x_2\stackrel{\eqref{eq:post-HAD-3}}{=} (\theta(x_1)\rhd S(x_2))(\theta(x_3)\rhd x_4)\\
&\stackrel{\eqref{anti-theta'}}{=}&
\theta(x_1)(\theta(x_2)\rhd x_3)\stackrel{\eqref{post-anti-coalg}, \eqref{eq:GL-product}}{=}
\theta(x_1)*_\rhd x_2\ \stackrel{\eqref{Post-con'}}{=}\  \iota(\varepsilon(\theta(x))),\\[.5em]
x_1*_\rhd (\theta(x_2)\rhd y) &\stackrel{\eqref{eq:GL-product}}{=}&
 x_1(x_2\rhd (\theta(x_3)\rhd y))
 \stackrel{\eqref{eq:post-HAD-4},\eqref{Post-con}}{=}
  x_1(\iota(\varepsilon(x_2))\rhd y)
 \stackrel{\eqref{eq:post-HAD-5}}{=} x_1\iota(\varepsilon(x_2)y \ =\  xy.
\end{eqnarray*}
Then taking $y=1_H$ in Eq.~\eqref{anti-theta''''} and applying Eq.~\eqref{eq:post-HAD-1'},  we see that $x_1*_\rhd \iota(\varepsilon(\theta(x_2)))=x$.
Hence, we have shown all the desired equalities.
\end{proof}

\begin{coro}
\label{coro:GL_Hopf_algebroid}
Given a post-Hopf algebroid $(H,\iota,\rhd,\theta)$ over $\lR$,
the Grossman-Larson bialgebroid $H_\rhd=(H,*_\rhd,\Delta,\varepsilon,\iota)$ equipped with $\theta$ forms a Hopf algebroid over $R$, which we call the {\bf Grossman-Larson Hopf algebroid} over $\lR$.
\end{coro}
\begin{proof}
Now it is clear that $\theta\iota=\iota$. So $H_\rhd$ is a Hopf algebroid over $R$ by Eqs.~\eqref{Post-con} and \eqref{Post-con'}.
\end{proof}

\begin{theorem}\label{thm:action-post-Hopf}
Let $(A,\rhd)$ be a cocommutative post-Hopf algebra and $\lR$ be an $A_\rhd$-module algebra via an action $\rightharpoonup$, where $A_\rhd=(A,*_\rhd,1_A,\Delta,\varepsilon,S_\rhd)$ is the sub-adjacent Hopf algebra of $(A,\rhd)$. Then the tensor product Hopf $\lR$-algebra $H=\lR\otimes A$ given in Example~\ref{ex:tensor-Hopf} has a natural post-Hopf algebroid structure $(H,\Delta_H,\varepsilon_H,S_H,\iota,\overline{\rhd},\theta)$ over $\lR$, where
\begin{eqnarray}
\label{action-post-Hopf-1}
fa\,\overline{\rhd}\, gb &=& f(a_1\rightharpoonup g)(a_2\rhd b),\\
\label{action-post-Hopf-2}
\theta(fa) &=& (S_\rhd(a_1)\rightharpoonup f)S_\rhd(a_2)
\end{eqnarray}
for any $f,g\in\lR$ and $a,b\in A$.
We call $(H,\iota,\overline{\rhd},\theta)$ an {\bf action post-Hopf algebroid} over $\lR$. In particular,
the corresponding Grossman-Larson Hopf algebroid $H_{\overline{\rhd}}=(H,*_{\overline{\rhd}},\Delta_H,\varepsilon_H,\theta,\iota)$ is the action Hopf algebroid $\lR\# A_\rhd$ defined as in Theorem~\ref{prop:smash-prod}.
\end{theorem}
\begin{proof}
First Eqs.~\eqref{eq:post-HAD-1}--\eqref{eq:post-HAD-1''} and \eqref{eq:post-HAD-2} are easy to see.
For any $f,g,h\in R$ and $a,b,c\in A$,
we have
\begin{eqnarray*}
\iota(\varepsilon_H(fa\,\overline{\rhd}\, gb)) &\stackrel{\eqref{action-post-Hopf-1}}{=}& \iota(\varepsilon_H(f(a_1\rightharpoonup g)(a_2\rhd b))) = \iota(f(a_1\rightharpoonup g)\varepsilon(a_2\rhd b))\\
&=& \iota(f(a_1\rightharpoonup g)\varepsilon(a_2)\varepsilon(b)) = \iota(f(a\rightharpoonup \varepsilon_H(gb)))
\stackrel{\eqref{action-post-Hopf-1}}{=} fa\,\overline{\rhd}\, \iota(\varepsilon_H(gb)),
\end{eqnarray*}
which implies that Eq.~\eqref{eq:post-HAD-2'} holds.

Eq.~\eqref{eq:post-HAD-3}   follows from
\begin{eqnarray*}
fa\,\overline{\rhd}\, (gb)(hc) &=&  fa\,\overline{\rhd}\, (gh)(bc)\\
&\stackrel{\eqref{action-post-Hopf-1}}{=}& f(a_1\rightharpoonup gh)(a_2\rhd bc)\\
&\stackrel{\eqref{Post-2}}{=}& f(a_1\rightharpoonup g)(a_2\rightharpoonup h)(a_3\rhd b)(a_4\rhd c)\\
&=& (f(a_1\rightharpoonup g)(a_2\rhd b))((a_3\rightharpoonup h)(a_4\rhd c))\\
&\stackrel{\eqref{action-post-Hopf-1}}{=}& (fa_1\,\overline{\rhd}\, gb)(a_2\,\overline{\rhd}\, hc).
\end{eqnarray*}
  Eq.~\eqref{eq:post-HAD-4} follows from
\begin{eqnarray*}
fa\,\overline{\rhd}\, (gb \,\overline{\rhd}\, hc) &\stackrel{\eqref{action-post-Hopf-1}}{=}& fa\,\overline{\rhd}\, g(b_1\rightharpoonup h)(b_2\rhd c)\\
&\stackrel{\eqref{action-post-Hopf-1}}{=}& f(a_1\rightharpoonup g(b_1\rightharpoonup h))(a_2\rhd (b_2\rhd  c))\\
&\stackrel{\eqref{Post-4},\eqref{post-rbb-1}}{=}& f(a_1\rightharpoonup g)(a_2\rightharpoonup (b_1\rightharpoonup h))
((a_3*_\rhd b_2)\rhd c)\\
&=& f(a_1\rightharpoonup g)((a_2*_\rhd b_1)\rightharpoonup h) ((a_3*_\rhd b_2)\rhd  c),\\
fa_1(a_2\,\overline{\rhd}\, gb) \,\overline{\rhd}\, hc &\stackrel{\eqref{action-post-Hopf-1}}{=}&  f(a_1\rightharpoonup g)a_2(a_3\rhd b)\,\overline{\rhd}\, hc\\
&\stackrel{\eqref{post-rbb-1}}{=}&  f(a_1\rightharpoonup g)(a_2*_\rhd b)\,\overline{\rhd}\, hc\\
&\stackrel{\eqref{action-post-Hopf-1}}{=}& f(a_1\rightharpoonup g)((a_2*_\rhd b_1)\rightharpoonup h) ((a_3*_\rhd b_2)\rhd  c).
\end{eqnarray*}
Hence, the tuple $(H,\iota,\overline{\rhd})$ is a weak post-Hopf algebroid over $\lR$.

Now we compute the Grossman-Larson product $*_{\overline{\rhd}}$ on $H$ as follows.
\begin{eqnarray*}
fa *_{\overline{\rhd}} gb &\stackrel{\eqref{eq:GL-product}}{=}& fa_1(a_2\,\overline{\rhd}\, gb)\ \stackrel{\eqref{action-post-Hopf-1}}{=}\
 f(a_1\rightharpoonup g)a_2(a_3\rhd b)
\ \stackrel{\eqref{post-rbb-1}}{=}\ f(a_1\rightharpoonup g)(a_2*_\rhd b).
\end{eqnarray*}
Next we check that the linear operator $\theta$ defined by \eqref{action-post-Hopf-2} is an algebra anti-involution of $H_{\overline{\rhd}}$.
\begin{eqnarray*}
\theta(fa)*_{\overline{\rhd}}\theta(gb) &\stackrel{\eqref{action-post-Hopf-2}}{=}& (S_\rhd(a_1)\rightharpoonup f)S_\rhd(a_2) *_{\overline{\rhd}}(S_\rhd(b_1)\rightharpoonup g)S_\rhd(b_2)\\
&=& (S_\rhd(a_1)\rightharpoonup f) (S_\rhd(a_2)\rightharpoonup(S_\rhd(b_1)\rightharpoonup g))(S_\rhd(a_3)*_\rhd S_\rhd(b_2))\\
&=& (S_\rhd(a_1)\rightharpoonup f)((S_\rhd(a_2)*_\rhd S_\rhd(b_1))\rightharpoonup g)(S_\rhd(a_3)*_\rhd S_\rhd(b_2))\\
&=&  (S_\rhd(a_1)\rightharpoonup f)(S_\rhd(b_1*_\rhd a_2)\rightharpoonup g)S_\rhd(b_2*_\rhd a_3),\\
\theta(gb*_{\overline{\rhd}} fa)
&=& \theta((b_1\rightharpoonup f)g(b_2*_\rhd a))\\
&\stackrel{\eqref{action-post-Hopf-2}}{=}& (S_\rhd(b_1*_\rhd a_1)\rightharpoonup (b_2\rightharpoonup f)g)S_\rhd(b_3*_\rhd a_2)\\
&=& (S_\rhd(b_1*_\rhd a_1)\rightharpoonup(b_2\rightharpoonup f))(S_\rhd(b_3*_\rhd a_2)\rightharpoonup g)S_\rhd(b_4*_\rhd a_3)\\
&=& ((S_\rhd(a_1)*_\rhd S_\rhd(b_1)*_\rhd b_2)\rightharpoonup f)(S_\rhd(b_3*_\rhd a_2)\rightharpoonup g)S_\rhd(b_4*_\rhd a_3)\\
&=& (S_\rhd(a_1)\rightharpoonup f)(S_\rhd(b_1*_\rhd a_2)\rightharpoonup g)S_\rhd(b_2*_\rhd a_3),\\[.5em]
\theta^2(fa) &\stackrel{\eqref{action-post-Hopf-2}}{=}& \theta((S_\rhd(a_1)\rightharpoonup f)S_\rhd(a_2))\\
&\stackrel{\eqref{action-post-Hopf-2}}{=}&(S^2_\rhd(a_1)\rightharpoonup (S_\rhd(a_2)\rightharpoonup f))S^2_\rhd(a_3)\\
&=&((a_1*_\rhd S_\rhd(a_2))\rightharpoonup f)a_3\\
&=& f\varepsilon(a_1)a_2 \ =\ fa.
\end{eqnarray*}
Meanwhile, we have
\begin{eqnarray*}
\Delta(\theta(fa)) &\stackrel{\eqref{action-post-Hopf-2}}{=}&
\Delta((S_\rhd(a_1)\rightharpoonup f)S_\rhd(a_2))\\
&=& (S_\rhd(a_1)\rightharpoonup f)S_\rhd(a_2)\otimes_\lR S_\rhd(a_3)\\
&\stackrel{\eqref{action-post-Hopf-2}}{=}& \theta(fa_1)\otimes_\lR \theta(a_2).\\[.5em]
fa_1*_{\overline{\rhd}} \theta(ga_2)
&\stackrel{\eqref{action-post-Hopf-2}}{=}& fa_1*_{\overline{\rhd}}(S_\rhd(a_2)\rightharpoonup g)S_\rhd(a_3)\\
&=& f(a_1 \rightharpoonup (S_\rhd(a_2)\rightharpoonup g))(a_3*_\rhd S_\rhd(a_4))\\
&=& f((a_1*_\rhd S_\rhd(a_2))\rightharpoonup g)(a_3*_\rhd S_\rhd(a_4))\\
&=& fg\varepsilon(a)\ =\ \varepsilon_H(fga).\\[.5em]
\theta(fa_1)*_{\overline{\rhd}} ga_2 &\stackrel{\eqref{action-post-Hopf-2}}{=}& (S_\rhd(a_1)\rightharpoonup f)S_\rhd(a_2)*_{\overline{\rhd}} ga_3\\
&=& (S_\rhd(a_1)\rightharpoonup f)(S_\rhd(a_2)\rightharpoonup g)(S_\rhd(a_3)*_\rhd a_4)\\
&=& (S_\rhd(a_1)\rightharpoonup fg)\varepsilon(S_\rhd(a_2))
\ =\ \varepsilon_H(\theta(fga)).
\end{eqnarray*}
Therefore, $\theta$ also satisfies all conditions~\eqref{post-anti-coalg}--\eqref{Post-con'}. The tensor product Hopf $\lR$-algebra $H=\lR\otimes A$ has the desired post-Hopf algebroid structure over $\lR$.
\end{proof}

\subsection{Braiding operators on Hopf algebroids}

Guccione, Guccione and Vendramin introduced the notion of a braiding operator on a cocommutative Hopf algebra in \cite[Definition~5.9]{GGV}.
Here we further generalize this notion to Hopf algebroids.
We interpret any Hopf algebroid $H$ over $\lR$ as an $\lR$-bimodule via the source map $\iota:\lR\to H$, and we use the notation $\bmot$ to denote the tensor product of $\lR$-bimodules, different from the former tensor product $\otimes_\lR$ of left $\lR$-modules.
\begin{defn}\label{def:braiding}
Given a Hopf algebroid $H$ over $\lR$, an $\lR$-bimodule homomorphism
$r:H\bmot H\to H\bmot H$ is called a {\bf braiding operator} on $H$ if it satisfies
\begin{enumerate}[(a)]
\item\label{eq:pb-coprod}
$(\id_H\otimes \tau\otimes\id_H)(\Delta\otimes \Delta)r = (r\otimes r)(\id_H\otimes \tau\otimes\id_H)(\Delta\otimes \Delta)$ as $\lR$-bimodule homomorphisms from $H\bmot H$ to $(H\bmot H)\otimes_\lR (H\bmot H)$,
\item\label{eq:pb-a}
$mr=m$,
\item\label{eq:pb-b}
$r(m\bmot\id_H)=(\id_H\bmot m)(r\bmot \id_H)(\id_H\bmot r)$,
\item\label{eq:pb-c}
$r(\id_H\bmot m)=(m\bmot\id_H)(\id_H\bmot r)(r\bmot \id_H)$,
\item\label{eq:pb-d}
$r(1_H\bmot x)=x\bmot 1_H,\ \forall x\in H$,
\item\label{eq:pb-e}
$r(x\bmot 1_H)=1_H\bmot x,\ \forall x\in H$,
\end{enumerate}
where $\tau$ is the flip map and $m$ is the multiplication of $H$, and
we can take the tensor product $\otimes_\lR$ since the $\lR$-bimodule $H\bmot H$ is a left $\lR$-module.
\end{defn}

\begin{lemma}\label{lem:pb-counit}
Given a braiding operator $r$ on a Hopf algebroid $H$ over $\lR$, we have
\begin{eqnarray}\label{eq:pb-counit}
\varepsilon( m(r(x\bmot y)))
&=&
\varepsilon(x \iota(\varepsilon(y))),\quad \forall x,y\in H.
\end{eqnarray}
\end{lemma}

\begin{proof}
Indeed, we have
\begin{eqnarray*}
\varepsilon( m(r(x\bmot y)))
&\stackrel{\ref{eq:pb-a}}{=}&
\varepsilon(xy)\ =\ \varepsilon(x\iota(\varepsilon(y)))
\end{eqnarray*}
for any $x,y\in H$.
\end{proof}

\begin{remark}
When $\lR=\bk$, the counit $\varepsilon:H\to \bk$ becomes an algebra map, and condition~\ref{eq:pb-coprod} and Eq.~\eqref{eq:pb-counit} imply that
$r$ is a coalgebra homomorphism, and thus a braiding operator on a Hopf algebra in the sense of \cite{GGV}.
\end{remark}

\begin{theorem}\label{thm:ph-braiding}
Let $(H,\iota,\rhd,\theta)$ be a post-Hopf algebroid over $\lR$.
There exists an $\lR$-bimodule homomorphism $r=r_{H_\rhd}:H_\rhd\bmot H_\rhd\to H_\rhd\bmot H_\rhd$ defined by
\begin{eqnarray}\label{eq:pha_ybo}
r(x\bmot y) &\coloneqq& x_1\rhd y_1\bmot \theta(x_2\rhd y_2)*_\rhd x_3*_\rhd y_3,
\end{eqnarray}
which  is a braiding operator on the Grossman-Larson Hopf algebroid $H_\rhd$ over $\lR$.
\end{theorem}
\begin{proof}
First, note that the map $r$ is well-defined, since the RHS of Eq.~\eqref{eq:pha_ybo} is independent with any expression of the coproduct $\Delta$ according to Eqs.~\eqref{post-anti-coef} and \eqref{anti-theta''''} when $x$ or  $y$ is of the form $\iota(f)*_\rhd z$ or $z*_\rhd\iota(f)$.
Moreover, $r$ is an $\lR$-bimodule homomorphism. That is,
\begin{eqnarray*}
r(\iota(f)*_\rhd x\bmot y) &=& \iota(f)*_\rhd(x_1\rhd y_1)\bmot \theta(x_2\rhd y_2)*_\rhd x_3*_\rhd y_3\ =\ \iota(f)*_\rhd r(x\bmot y) ,\\
r(x\bmot y*_\rhd\iota(f)) &=& x_1\rhd y_1\bmot \theta(x_2\rhd y_2)*_\rhd x_3*_\rhd y_3*_\rhd\iota(f)\ =\ r(x\bmot y)*_\rhd\iota(f).
\end{eqnarray*}

We then check that $r$ is a braiding operator on $H_\rhd$.
First note that both hand-sides of the equality in condition~\ref{eq:pb-coprod} are well-defined $\lR$-bimodule homomorphisms,
and then such an equality is due to Eqs.~\eqref{eq:post-HAD-1} and \eqref{post-anti-coalg}.
Condition~\ref{eq:pb-a} is clear by Eqs.~\eqref{eq:pha_ybo}, \eqref{eq:post-HAD-1} and \eqref{Post-con}.
Conditions~\ref{eq:pb-b} and \ref{eq:pb-c} follows from
\begin{eqnarray*}
&&(\id\bmot *_\rhd)(r\bmot \id)(\id\bmot r)(x\bmot y\bmot z)\\
&\stackrel{\eqref{eq:post-HAD-1},\eqref{eq:pha_ybo}}{=}& x_1\rhd(y_1\rhd z_1)\bmot \theta(x_2\rhd(y_2\rhd z_2))*_\rhd x_3*_\rhd(y_3\rhd z_3)*_\rhd  \theta(y_4\rhd z_4)*_\rhd y_5*_\rhd z_5\\
&\stackrel{\eqref{eq:post-HAD-1},\eqref{Post-con}}{=}&
x_1\rhd(y_1\rhd z_1)\bmot \theta(x_2\rhd(y_2\rhd z_2))*_\rhd x_3*_\rhd y_3*_\rhd z_3\\
&\stackrel{\eqref{eq:post-HAD-4},\eqref{eq:GL-product}}{=}&
(x_1*_\rhd y_1)\rhd z_1 \bmot \theta ((x_2*_\rhd y_2)\rhd z_2)*_\rhd x_3*_\rhd y_3*_\rhd z_3\\
&\stackrel{\eqref{eq:pha_ybo}}{=}& r(x*_\rhd y \bmot z),\\[.5em]
&& (*_\rhd\bmot \id)(\id\bmot r)(r\bmot \id)(x\bmot y\bmot z)\\
&\stackrel{\eqref{eq:post-HAD-1},\eqref{eq:pha_ybo}}{=}&  (x_1\rhd y_1)*_\rhd((\theta(x_2\rhd y_2)*_\rhd x_3*_\rhd y_3) \rhd z_1)
\bmot \\
&&\qquad \theta((\theta(x_4\rhd y_4)*_\rhd x_5*_\rhd y_5) \rhd z_2)*_\rhd
\theta(x_6\rhd y_6)*_\rhd x_7*_\rhd y_7*_\rhd z_3 \\
&\stackrel{\eqref{eq:post-HAD-4},\eqref{eq:GL-product}}{=}&  (x_1\rhd y_1)*_\rhd(\theta(x_2\rhd y_2)\rhd ((x_3*_\rhd y_3) \rhd z_1))
\bmot \\
&&\qquad \theta((x_4\rhd y_4)*_\rhd (\theta(x_5\rhd y_5)\rhd ((x_6*_\rhd y_6) \rhd z_2)))*_\rhd x_7*_\rhd y_7*_\rhd z_3 \\
&\stackrel{\eqref{eq:post-HAD-1},\eqref{anti-theta''''}}{=}&
(x_1\rhd y_1)((x_2*_\rhd y_2) \rhd z_1)
\bmot \theta((x_3\rhd y_3)((x_4*_\rhd y_4) \rhd z_2))*_\rhd x_5*_\rhd y_5*_\rhd z_3 \\
&\stackrel{\eqref{eq:post-HAD-4},\eqref{eq:GL-product}}{=}&
(x_1\rhd y_1)(x_2\rhd (y_2 \rhd z_1))
\bmot \theta((x_3\rhd y_3)(x_4\rhd (y_4 \rhd z_2)))*_\rhd x_5*_\rhd y_5*_\rhd z_3 \\
&\stackrel{\eqref{eq:post-HAD-3},\eqref{eq:GL-product}}{=}&
x_1\rhd (y_1*_\rhd z_1)
\bmot \theta(x_2\rhd (y_2 *_\rhd z_2))*_\rhd x_3*_\rhd y_3*_\rhd z_3 \\
&\stackrel{\eqref{eq:pha_ybo}}{=}& r(x\bmot y*_\rhd z).
\end{eqnarray*}
Conditions~\ref{eq:pb-d} and \ref{eq:pb-e} follows from
\begin{eqnarray*}
r(1_H \bmot x) &\stackrel{\eqref{eq:post-HAD-1''},\eqref{eq:pha_ybo}}{=}&  x_1\bmot \theta(x_2)*_\rhd x_3\stackrel{\eqref{Post-con'}}{=}
x_1\bmot \iota(\varepsilon(\theta(x_2)))\\
&=& x_1*_\rhd \iota(\varepsilon(\theta(x_2)))\bmot 1_H\stackrel{\eqref{anti-theta'''}}{=} x\bmot 1_H,\\[.5em]
r(x \bmot 1_H) &\stackrel{\eqref{eq:pha_ybo}}{=}&
x_1\rhd 1_H \bmot \theta(x_2\rhd 1_H)*_\rhd x_3\\
&\stackrel{\eqref{eq:post-HAD-1'}}{=}& \iota(\varepsilon(x_1)) \bmot \iota(\varepsilon(x_2))*_\rhd x_3\\
&=& 1_H \bmot \iota(\varepsilon(x_1))*_\rhd x_2 \ =\ 1_H\bmot x.
\end{eqnarray*}
All the braiding conditions have been checked. The proof is finished.
\end{proof}

\section{The universal enveloping algebras of post-Lie-Rinehart algebras}\label{sec:post-LR}

 In this section, first we recall the universal enveloping algebra of a Lie-Rinehart algebra. Then we introduce the notion of post-Lie-Rinehart algebras as the algebraic analog of post-Lie algebroids. We construct the free post-Lie-Rinehart algebra generated by a magma algebra with a fixed linear map to $\Der(R)$, which turns out to be an action post-Lie-Rinehart algebra. Finally we   show that the universal enveloping algebra of a post-Lie-Rinehart algebra  is a weak post-Hopf algebroid. In the particular case of action post-Lie-Rinehart algebras, the universal enveloping algebras are post-Hopf algebroids.

\subsection{Lie-Rinehart algebras and their universal enveloping algebras}

The structure of a Lie-Rinehart algebra appeared in Rinehart's work \cite{Ri} under the terminology of $(R,A)$-Lie algebra. Then, Huebschmann referred to such a structure a Lie-Rinehart algebra \cite{Hu90,Hu04}. Lie-Rinehart algebras are the algebraic counterpart of Lie algebroids, and have various important applications in Poisson geometry, symplectic geometry and quantization.

\begin{defn}\cite{Ri}
Let $\lL$ be an $\lR$-module.
A tuple $\mathscr{\lR}=(\lR,\lL,[-,-]_\lL,\rho)$ is called a {\bf Lie-Rinehart algebra}, if
\begin{enumerate}[(i)]
\item
$\lL$ is a Lie algebra with its bracket $[-,-]_\lL$.
\item
$\rho:\lL\to\Der(\lR)$ is an $\lR$-module homomorphism and a Lie algebra homomorphism, called the {\bf anchor}.
\item
the following compatibility condition holds:
\begin{eqnarray}\label{eq:LR}
[x,fy]_\lL &=& f[x,y]_\lL+\rho(x)(f)y,\quad\forall f\in \lR,\,x,y\in \lL.
\end{eqnarray}
\end{enumerate}
By convenience, we will write $\rho(x)(f)$ as $x(f)$ for short.
\end{defn}

\begin{exam}
Recall that a {\bf Lie algebroid} structure on a vector bundle $A\longrightarrow M$ is
a pair consisting of a Lie algebra structure $[-,-]_A$ on
the section space $\Gamma(A)$ and a vector bundle morphism
$a_A:A\longrightarrow TM$
from $A$ to the tangent bundle $TM$, called the {\bf anchor}, satisfying
\begin{equation}\label{eq:LAf}
[x,fy]_A=f[x,y]_A+a_A(x)(f)y,\quad \forall f\in
\CWM, x,y\in\Gamma(A),
\end{equation}
where $\CWM$ is the commutative associative algebra of smooth functions on the manifold $M$. Let $(A\longrightarrow M,[-,-]_A,a_A)$ be a Lie algebroid. Then the quadruple $(\CWM,\Gamma(A),[-,-]_A,a_A)$ is a Lie-Rinehart algebra.
\end{exam}

\begin{exam}
For the Lie algebra $\Der(\lR)$ of derivations on $\lR$, the triple $(\lR,\Der(\lR),\id)$ is a Lie-Rinehart algebra.
\end{exam}

\begin{exam}\label{exam:action-LR}
Let $\frakg$ be a Lie algebra with a Lie algebra homomorphism $\rho:\frakg\to \Der(\lR)$. Let $L=\lR\otimes\frakg$ as a free $\lR$-module. Define
\begin{eqnarray*}
[fx,gy]_\lL &=& fg[x,y]_\frakg+f\rho(x)(g)y-g\rho(y)(f)x,\quad\forall f,g\in \lR,\,x,y\in \frakg,
\end{eqnarray*}
and extend $\rho$ to be an $\lR$-module homomorphism $\overline{\rho}:\lL\to\Der(\lR)$. Then
the quadruple $(\lR,\lL,[-,-]_\lL,\overline{\rho})$ is a Lie-Rinehart algebra, and we call it an {\bf action Lie-Rinehart algebra}.
\end{exam}

\begin{defn}
A Lie algebra $\lL$ is called a {\bf Lie $\lR$-algebra} if $\lL$ is an $\lR$-module and its Lie bracket $[-,-]_\lL$ is $\lR$-linear. In other words, a Lie $\lR$-algebra is a Lie-Rinehart algebra with trivial anchor.
\end{defn}

Note that the authors in \cite{MM} defined the universal enveloping algebra of a Lie algebroid $(A\longrightarrow M,[-,-]_A,a_A)$ as the universal enveloping algebra of its underlying Lie-Rinehart algebra $(\CWM,\Gamma(A),[-,-]_A,a_A)$.
Now we recall the definition of the universal enveloping algebra of a Lie-Rinehart algebra.
\begin{defn}[\cite{MM}]
For a Lie-Rinehart algebra $\mathscr{\lR}=(\lR,\lL,[-,-]_\lL,\rho)$, the {\bf  universal enveloping algebra} of $\mathscr{\lR}$ is a triple $(\mathscr U_\lR(\lL),\iota_\lR,\iota_\lL)$, where $\mathscr U_\lR(\lL)$ is a unital algebra, $\iota_\lR:\lR\to \mathscr U_\lR(\lL)$ is a homomorphism of unital algebras and $\iota_\lL:\lL\to \mathscr U_\lR(\lL)$ is a homomorphism of Lie algebras
such that
$$\iota_\lR(f)\iota_\lL(x)=\iota_\lL(fx),\quad \iota_\lL(x)\iota_\lR(f)-\iota_\lR(f)\iota_\lL(x)=\iota_\lR(x(f)),\quad \forall f\in\lR, x\in \lL,$$
and satisfying the following universal property: if $A$ is a unital algebra, $\kappa_\lR:\lR\to A$ is a homomorphism of unital algebras and $\kappa_\lL:\lL\to A$ is a homomorphism of Lie algebras such that
$$\kappa_\lR(f)\kappa_\lL(x)=\kappa_\lL(fx),\quad \kappa_\lL(x)\kappa_\lR(f)-\kappa_\lR(f)\kappa_\lL(x)=\kappa_\lR(x(f)),\quad \forall f\in\lR, x\in \lL,$$
then there exists a unique homomorphism of unital algebras $\bar\kappa:\mathscr U_\lR(\lL)\to A$ such that $\bar\kappa \iota_\lR=\kappa_\lR$ and $\bar\kappa \iota_\lL=\kappa_\lL$.
\end{defn}

We have the following construction of the universal enveloping algebra of a Lie-Rinehart algebra.
\begin{prop}[{\cite[\S~3.2]{KS}}]
For a Lie-Rinehart algebra $\mathscr{\lR}=(\lR,\lL,[-,-]_\lL,\rho)$, let $T_\lR(\lL\otimes \lR)=\lR \oplus \bigoplus_{n\geq1}(\lL\otimes \lR)^{\otimes_\lR n}$ be the tensor $\lR$-algebra over the $\lR$-bimodule $\lL\otimes \lR$, and let $j_\lR:\lR\to T_\lR(\lL\otimes \lR)$ and $j_\lL:\lL\to T_\lR(\lL\otimes \lR)$ be the canonical embedding, i.e.
$$j_\lR(f)=f,\quad  j_\lL(x)\mapsto x\otimes 1_\lR.$$
 Define
\begin{equation}\label{eq:uea-LR}
\mathscr U_\lR(\lL)=T_\lR(\lL\otimes \lR)\bigg/\left({j_\lL(x)j_\lL(y)-j_\lL(y)j_\lL(x)-j_\lL([x,y]_\lL), \atop j_\lL(x)j_\lR(f) - j_\lR(f)j_\lL(x) - j_\lR(x(f))}\,\bigg|\,{x,y\in \lL,\atop f\in \lR}\right),
\end{equation}
and $j_\lR$ (resp. $j_\lL$) induces the map $\iota_\lR:\lR\to \mathscr U_\lR(\lL)$
(resp. $\iota_\lL:\lL\to \mathscr U_\lR(\lL)$). Then $(\mathscr U_\lR(\lL),\iota_\lR,\iota_\lL)$ is the universal enveloping algebra of $\mathscr{\lR}$.
\end{prop}

\begin{remark}
The structural map $\iota_\lR:\lR\to \mathscr U_\lR(\lL)$ is injective, while
$\iota_\lL:\lL\to \mathscr U_\lR(\lL)$ is also injective if $\lL$ is $\lR$-projective, by the PBW theorem proven in~\cite[Theorem~3.1]{Ri}.
\end{remark}

For the universal enveloping algebra $\mathscr U_\lR(\lL)$ of a Lie-Rinehart algebra $\mathscr{\lR}=(\lR,\lL,[-,-]_\lL,\rho)$, using its left $\lR$-module structure to define the tensor product $\otimes_\lR$, then we have the comultiplication  as a left $\lR$-module map and a $\bk$-algebra homomorphism
$$\Delta:\mathscr U_\lR(\lL)\to \mathscr U_\lR(\lL)\,\bar{\otimes}_\lR\, \mathscr U_\lR(\lL),\ f\mapsto f\otimes_\lR 1_\lR,\ x\mapsto x\otimes_\lR 1_\lR + 1_\lR\otimes_\lR x,\quad \forall f\in \lR,\,x\in\lL,$$
where $\mathscr U_\lR(\lL)\,\bar{\otimes}_\lR\, \mathscr U_\lR(\lL)
$ is the kernel of the map
$$\vartheta:\mathscr U_\lR(\lL)\otimes_\lR \mathscr U_\lR(\lL)\to \Hom(\lR,\mathscr U_\lR(\lL)\otimes_\lR \mathscr U_\lR(\lL)),\ X\otimes_\lR Y\mapsto(f\mapsto Xf\otimes_\lR Y- X\otimes_\lR Yf).$$
The counit  $\varepsilon: \mathscr U_\lR(\lL) \to \lR$ is the left $\lR$-module map such that
$$\varepsilon(f)=f,\quad \varepsilon(x)=0,\quad \varepsilon(XY)=\varepsilon(X\varepsilon(Y)),\quad \forall f\in \lR,\,x\in\lL,\,X,Y\in\mathscr U_\lR(\lL).$$
Then $(\mathscr U_\lR(\lL),\cdot,\Delta,\varepsilon,\iota_\lR)$ is a bialgebroid over $\lR$, also called an {\bf $\lR/\bk$-bialgebra} in~\cite{MM} or a {\bf left bialgebroid} over $\lR$ in~\cite{BKS}.

\begin{prop}\label{prop:uea-LRA}
Let $\lL$ be a Lie $\lR$-algebra with Lie bracket $[-,-]_\lL$. Then the universal enveloping algebra $\mathscr U_\lR(\lL)$ of $L$ is a Hopf $R$-algebra.
\end{prop}
\begin{proof}
Note that a Lie $\lR$-algebra is a Lie-Rinehart algebra with trivial anchor.
So $\iota_\lR(\lR)$ lies in the center of $\mathscr U_\lR(\lL)$ according to Eq.~\eqref{eq:uea-LR}.
Let $T_\lR(\lL)=\lR \oplus \bigoplus_{n\geq1}\lL^{\otimes_\lR n}$ be the tensor $\lR$-algebra over the $\lR$-module $\lL$ and let $j_\lL:\lL\to T_\lR(\lL)$
 be the canonical embedding. We have the following $\lR$-algebra isomorphism
$$\mathscr U_\lR(\lL)\cong T_\lR(\lL)/\left(j_\lL(x)j_\lL(y)-j_\lL(y)j_\lL(x)-j_\lL([x,y]_\lL)\,|\,x,y\in \lL\right).$$
Hence, we can define the antipode map
$$S: \mathscr U_\lR(\lL) \to \mathscr U_\lR(\lL),\ \ f\mapsto f,\ x\mapsto -x,\quad \forall f\in \lR,\,x\in\lL$$
as an $\lR$-algebra anti-isomorphism, and $\mathscr U_\lR(\lL)$ becomes a Hopf $R$-algebra.
\end{proof}

As a special example of action Hopf algebroids defined in Theorem~\ref{prop:smash-prod}, we have
\begin{exam}\label{ex:uea_action_LR}
Let $\mathscr{\lR}$ be the action Lie-Rinehart algebra $(\lR,\lL=\lR\otimes \frakg,\rho)$ given in Example~\ref{exam:action-LR}. Since $\rho$ naturally defines a module algebra action $\rightharpoonup$ of $\mathscr U(\frakg)$ over $\lR$, the universal enveloping algebra
$(\mathscr U_\lR(\lL),\iota_\lR,\iota_\lL)$ of $\mathscr{\lR}$
is a cocommutative Hopf algebroid over $\lR$ isomorphic to the action Hopf algebroid $\lR\# \mathscr U(\frakg)$ over $\lR$.
\end{exam}

\subsection{Post-Lie-Rinehart algebras and free objects}
\label{sec:free_Post-Lie-Rinehart}
In this subsection, we introduce the notion of a post-Lie-Rinehart algebra as the algebraic counterpart of a post-Lie algebroid introduced in \cite{ML}, and study its free object.

The notion of a post-Lie algebra was introduced by Vallette from his study of Koszul duality of operads in \cite{Val}.

\begin{defn}\cite{Val}
A {\bf post-Lie algebra} $(\lL,[-,-]_\lL,\rhd)$ consists of a Lie algebra $(\lL,[-,-]_\lL)$ and a binary product $\rhd:\lL\otimes\lL\to\lL$ such that
\begin{eqnarray}
\label{Post-L-1}x\rhd[y,z]_\lL&=&[x\rhd y,z]_\lL+[y,x\rhd z]_\lL,\\
\label{Post-L-2}([x,y]_\lL+x\rhd y-y\rhd x)\rhd z&=&x\rhd(y\rhd z)-y\rhd(x\rhd z).
\end{eqnarray}

A {\bf post-Lie algebra homomorphism} between two post-Lie algebras is a Lie algebra homomorphism compatible with post-Lie products.
 \end{defn}

Any post-Lie algebra $(\lL,[-,-]_\lL,\rhd)$ has a {\bf sub-adjacent Lie algebra}
$\lL_\rhd:=(\lL,[-,-]_{ \rhd})$
defined by
$$[x,y]_\rhd \coloneqq x\rhd y-y\rhd x+[x,y]_\lL,\quad\forall x,y\in\lL.$$

The notion of a post-Lie algebroid was introduced by Munthe-Kaas and Lundervold in their study of numerical integration \cite{ML}. Now we introduce the notion of a post-Lie-Rinehart algebra. The relationship between post-Lie-Rinehart algebras and post-Lie algebroids is similar to the relationship between Lie-Rinehart algebras and Lie algebroids.

\begin{defn}
Let $\lL$ be a Lie $\lR$-algebra with Lie bracket $[-,-]_\lL$ and $\rhd:L\otimes L\to L$ be a bilinear map.
A tuple $\mathscr{\lR}=(\lR,\lL,[-,-]_\lL,\rho,\rhd)$ is called a {\bf post-Lie-Rinehart algebra}, if
\begin{enumerate}[(i)]
\item
$(\lL,[-,-]_\lL,\rhd)$ is a post-Lie algebra,
\item
$\rho:\lL_\rhd\to\Der(\lR)$ is an $\lR$-module homomorphism and a Lie algebra homomorphism for the sub-adjacent Lie algebra $\lL_\rhd$ of $(L,[-,-]_\lL,\rhd)$,
\item
the following compatibility conditions hold:
\begin{eqnarray}
\label{eq:PLR-1}
fx\rhd y &=& f(x\rhd y),\\
\label{eq:PLR-2}
x\rhd fy &=& f(x\rhd y)+x(f)y
\end{eqnarray}
for any $f\in \lR,\,x,y\in \lL$, where $x(f)=\rho(x)(f)$ as before.
\end{enumerate}
A {\bf post-Lie-Rinehart algebra homomorphism} between two post-Lie-Rinehart algebras is an $\lR$-module homomorphism and a post-Lie algebra homomorphism compatible with anchors.
\end{defn}

\begin{prop}
Given a post-Lie-Rinehart algebra $\mathscr{\lR}=(\lR,\lL,[-,-]_\lL,\rho,\rhd)$, the tuple $\mathscr{\lR}_\rhd=(\lR,\lL,[-,-]_{\rhd},\rho)$ is a Lie-Rinehart algebra, which is called the {\bf sub-adjacent Lie-Rinehart algebra} of $\mathscr{\lR}$.
\end{prop}
\begin{proof}
By definition, we only need to check the compatibility condition~\eqref{eq:LR} as follows:
\begin{eqnarray*}
[x,fy]_{\rhd} &=& x \rhd fy - fy \rhd x + [x,fy]_L\\
&\stackrel{\eqref{eq:PLR-1},\eqref{eq:PLR-2}}{=}& f(x\rhd y) +x(f)y - f(y\rhd x)  + f[x,y]_L\\
&=& f[x,y]_{\rhd}+x(f)y
\end{eqnarray*}
for any $f\in \lR$ and $x,y\in \lL$.
\end{proof}

\begin{exam}\label{ex:post-Lie-Rinehart}
  Recall from \cite{ML} that a {\bf post-Lie algebroid} structure on a vector bundle
$A\longrightarrow M$ is a triple that consists of a $C^\infty(M)$-linear Lie
algebra structure $[-,-]_A$ on $\Gamma(A)$, a  bilinear operation   $\triangleright_A:\Gamma(A)\otimes \Gamma(A)\longrightarrow \Gamma(A)$   and a
vector bundle morphism $a_A:A\longrightarrow TM$, called the {\bf anchor},
such that $(\Gamma(A),[-,-]_A,\triangleright_A)$ is a post-Lie algebra, and
for all $f\in\CWM$ and $u,v\in\Gamma(A)$, the following
relations are satisfied:
\begin{itemize}
\item[\rm(i)]$u\triangleright_A(fv)=f(u\triangleright_A v)+a_A(u)(f)v,$
\item[\rm(ii)] $(fu)\triangleright_A v=f(u\triangleright_A v).$
\end{itemize}
It is straightforward to check that $(\CWM,\Gamma(A),[-,-]_A,a_A,\triangleright_A)$ is a post-Lie-Rinehart algebra.
\end{exam}

As a generalization of \cite[Theorem 2.7]{LST}, there is a post-Lie-Rinehart algebra structure on the space of primitive elements of a weak post-Hopf algebroid.

Given a weak post-Hopf algebroid $(H,\iota,\rhd)$ over $\lR$, let
\begin{eqnarray*}
P(H) &\coloneqq& \{x\in H\,|\,\Delta(x)=x\otimes_\lR 1_H + 1_H\otimes_\lR x\}.
\end{eqnarray*}
Since $H$ is a Hopf $\lR$-algebra, it is straightforward to see that $(P(H),[-,-])$ is a Lie $R$-algebra with the commutator as its Lie bracket. By the same argument as \cite[Theorem 2.7]{LST}, the product $\rhd$ can be restricted to $P(H)$ such that $(P(H),[-,-],\rhd)$ is a post-Lie algebra. Moreover, we have the following result.

\begin{prop}
Let $(H,\iota,\rhd)$ be a weak post-Hopf algebroid  over $\lR$.
Then $(\lR,P(H),[-,-],\rho,\rhd)$ is a post-Lie-Rinehart algebra, where
the anchor $\rho$ is given by~ Eq. \eqref{eq:left-action}, i.e.
\begin{eqnarray*}
\rho(x)(f) &=& \varepsilon(x\rhd \iota(f)),\quad f\in\lR,\ x\in P(H).
\end{eqnarray*}
\end{prop}
\begin{proof}
First note $\rho(x)\in\Der(\lR)$ for any $x\in P(H)$ by conditions \eqref{eq:post-HAD-1''} and \eqref{eq:post-HAD-3}.
It remains to show that $\rho$ is an $\lR$-module homomorphism and a Lie algebra homomorphism, and Eqs.~\eqref{eq:PLR-1} and \eqref{eq:PLR-2} hold. The fact that $\rho$ is an $\lR$-module homomorphism follows from
$$
\rho(fx)(g) = \varepsilon(fx\rhd \iota(g))
\ \stackrel{\eqref{eq:post-HAD-2}}{=}\ f\varepsilon(x\rhd \iota(g))\ =\ f\rho(x)(g),\quad \forall f, g\in \lR, ~x\in P(H).
$$
For all $f\in\lR$ and $x,y\in P(H)$, we have
\begin{eqnarray*}
{[\rho(x),\rho(y)](f)} &=& \rho(x)(\rho(y)(f)) - \rho(y)(\rho(x)(f)) \\
&=& \varepsilon(x\rhd \iota(\varepsilon(y\rhd \iota(f)))) - \varepsilon(y\rhd \iota(\varepsilon(x\rhd \iota(f))))\\
&\stackrel{\eqref{eq:post-HAD-2'}}{=}& \varepsilon(x\rhd (y\rhd \iota(f))) - \varepsilon(y\rhd (x\rhd \iota(f)))\\
&\stackrel{\eqref{eq:post-HAD-1''},\,\eqref{eq:post-HAD-4}}{=}& \varepsilon((xy+x\rhd y)\rhd \iota(f)) - \varepsilon((yx+y\rhd x)\rhd \iota(f))\\
&=& \varepsilon(([x,y] + x\rhd y - y\rhd x) \rhd \iota(f))\\
&=& \varepsilon([x,y]_\rhd \rhd \iota(f))\ \\
&=&\rho([x,y]_\rhd)(f),
\end{eqnarray*}
which implies that $\rho$ is  a Lie algebra homomorphism.

On the other hand, Eq.~\eqref{eq:post-HAD-2} clearly implies Eq.~\eqref{eq:PLR-1}. Also, we have
\begin{eqnarray*}
x\rhd fy \stackrel{\eqref{eq:post-HAD-3}}{=} (x\rhd \iota(f))y + \iota(f)(x\rhd y) \stackrel{\eqref{eq:post-HAD-5'}}{=}  f(x \rhd y) +
\varepsilon(x\rhd \iota(f))y = f(x \rhd y) + \rho(x)(f)y.
\end{eqnarray*}
Namely, Eq.~\eqref{eq:PLR-2} holds.
Hence, $(\lR,P(H),[-,-],\rho,\rhd)$ is a post-Lie-Rinehart algebra.
\end{proof}

Next we introduce a useful construction of post-Lie-Rinehart algebras, originally given in \cite{GLS} via a geometric terminology, namely action post-Lie algebroids.
\begin{prop}\label{thm:post-LR-derivation}
Let $(\frakg,[-,-]_\frakg,\rhd)$ be a post-Lie algebra with its sub-adjacent Lie algebra $\frakg_\rhd$. Let $\rho:\frakg_\rhd\to \Der(\lR)$ be a Lie algebra homomorphism. Then we have an extended post-Lie algebra structure on $\lL=\lR\otimes \frakg$ defined as follows,
\begin{eqnarray}
\label{eq:action-post-LR-1}
[f_1x_1, f_2x_2]_\lL &=& f_1f_2[x_1,x_2]_\frakg,\\
\label{eq:action-post-LR-2}
f_1x_1 \,\overline{\rhd}\, f_2x_2 &=& f_1x_1(f_2)x_2 + f_1f_2(x_1\rhd x_2)
\end{eqnarray}
for any $f_1,f_2\in \lR$ and $x_1,x_2\in \frakg$, where we write $\rho(x)(f)$ simply as $x(f)$.

Extend $\rho:\frakg\to\Der(\lR)$ to be a linear map  $\overline{\rho}:\lL\to\Der(\lR)$ defined by
\begin{eqnarray*}
\overline{\rho}(fx)(f')&=& fx(f'),\quad f,f'\in \lR,\ x\in \frakg.
\end{eqnarray*}
Then, the tuple $(\lR,\lL,[-,-]_\lL,\overline{\rho},\overline{\rhd})$ is a post-Lie-Rinehart algebra, and we call it an {\bf action post-Lie-Rinehart algebra}.

The corresponding sub-adjacent Lie-Rinehart algebra $\lL_{\overline{\rhd}}\coloneqq(\lR,\lL,[-,-]_{\overline{\rhd}},\overline{\rho})$ is given by
\begin{eqnarray*}
[f_1x_1, f_2x_2]_{\overline{\rhd}} &=& f_1x_1\,\overline{\rhd}\, f_2x_2 - f_2x_2\,\overline{\rhd}\, f_1x_1+[f_1x_1,f_2x_2]_\lL\\
&=& f_1x_1(f_2)x_2 - f_2x_2(f_1)x_1 + f_1f_2[x_1,x_2]_{\rhd}.
\end{eqnarray*}
\end{prop}

\begin{remark}
The special case of Proposition~\ref{thm:post-LR-derivation} for a Lie algebra $\frakg$ with zero post-Lie product generalizes the post-Lie algebras of derivations in {\cite[Theorem~3.1]{JZ}}.
\end{remark}

At the end of this subsection, we construct the free post-Lie-Rinehart algebra generated by a magma algebra with a fixed linear map to $\Der(\lR)$.

Let $(V,\curvearrowright)$ be a magma algebra. It is well-known that the free post-Lie algebra ${\rm PostLie}(V)$ generated by $(V,\curvearrowright)$ is the free Lie algebra $({\rm Lie}(V),[-,-]_{{\rm Lie}(V)})$ equipped with the extended post-Lie product $\rhd_V$ from $\curvearrowright$~\cite{Fo}. Namely, for any post-Lie algebra $(\frakg,[-,-]_\frakg,\rhd)$ with a magma algebra homomorphism $\phi:(V,\curvearrowright)\to(\frakg,\rhd)$, there exists a unique post-Lie algebra homomorphism $ \tilde{\phi}:({\rm Lie}(V),[-,-]_{{\rm Lie}(V)},\rhd_V)\to (\frakg,[-,-]_\frakg,\rhd)$ such that $\phi=\tilde\phi i_V$, where $i_V:V\to {\rm Lie}(V)$ is the natural inclusion.

In order to define a free post-Lie-Rinehart algebra generated by $(V,\curvearrowright)$, a linear map $f_V:V\to \Der(\lR)$
need to be given first. Now we extend it to be a Lie algebra homomorphism
$\rho_V:{\rm PostLie}(V)_{\rhd_V}\to \Der(\lR)$ by the following recursion,
\begin{eqnarray*}
\rho_V(x) &=& f_V(x),\\
\rho_V([x,y]_{{\rm Lie}(V)}) &=&[f_V(x),f_V(y)]-f_V(x\curvearrowright y)+f_V(y\curvearrowright x),\\
\rho_V([x,X]_{{\rm Lie}(V)}) &=& [f_V(x),\rho_V(X)]- \rho_V(x\rhd_V X) + f_V(X\rhd_V x),
\end{eqnarray*}
for any $x,y\in V$ and $X\in {\rm Lie}(V)$.

By Proposition~\ref{thm:post-LR-derivation}, we have the action post-Lie-Rinehart algebra
$${\rm PostLR}(V)\coloneqq (\lR,\lR\otimes {\rm Lie}(V),[-,-]_{\lR\otimes {\rm Lie}(V)},\overline{\rho}_V,\overline{\rhd}_V).$$
\begin{theorem}\label{thm:free-post-LR}
The action post-Lie-Rinehart algebra
${\rm PostLR}(V)$ has the following universal property:

For any post-Lie-Rinehart algebra $\mathscr{\lR}=(\lR,\lL,[-,-]_\lL,\rho,\rhd)$, and a magma algebra homomorphism $\varphi:(V,\curvearrowright)\to (\lL,\rhd)$ such that $\rho\varphi=f_V$,
there exists a unique post-Lie-Rinehart algebra homomorphism $\tilde\varphi:{\rm PostLR}(V)\to L$ such that $\varphi=\tilde\varphi (1_\lR\otimes i_V)$, where $i_V:V
\to {\rm Lie}(V)$ is the natural embedding. That is, we have the following commutative diagram:
$$\xymatrix@=2em{\Der(\lR) & V\ar@{->}[d]_-{\varphi}\ar@{->}[r]^-{1_\lR\otimes i_V}\ar@{->}[l]_-{f_V}&{\lR\otimes {\rm Lie}(V)}\ar@{.>}[dl]^-{\exists!\, \tilde\varphi}\\
	& \lL \ar@{->}[ul]^-{\rho}&}$$
Thus, the post-Lie-Rinehart algebra ${\rm PostLR}(V)$ is the {\bf free post-Lie-Rinehart algebra} generated by $(V,\curvearrowright,f_V)$.
\end{theorem}
\begin{proof}
First by the universal property of ${\rm PostLie}(V)$, there exists a unique post-Lie algebra homomorphism $\psi:{\rm PostLie}(V)\to (\lL,[-,-]_\lL,\rhd)$ such that $\varphi=\psi i_V$. It induces a Lie algebra homomorphism $\rho\psi:{\rm PostLie}(V)_{\rhd_V}\to\Der(\lR)$ satisfying
$\rho\psi i_V=\rho\varphi=f_V,$
so we have $\rho\psi=\rho_V$, since ${\rm Lie}(V)$ is generated by $V$.

Now define a linear map $\tilde\varphi:\lR\otimes {\rm Lie}(V)\to \lL$ by
$$\tilde\varphi(fX)=f\psi(X),\quad\forall f\in\lR,X\in {\rm Lie}(V).$$
It is clear that $\tilde\varphi$ is $\lR$-linear and $\tilde\varphi (1_\lR\otimes i_V)=\psi i_V=\varphi$. Also, we have
\begin{eqnarray*}
\tilde\varphi([fX,gY]_{\lR\otimes {\rm Lie}(V)}) &=& \tilde\varphi(fg[X,Y]_{{\rm Lie}(V)})\ =\ fg\psi([X,Y]_{{\rm Lie}(V)})\\
&=& [f\psi(X),g\psi(Y)]_\lL\ =\
[\tilde\varphi(fX),\tilde\varphi(gY)]_\lL,\\[.5em]
\tilde\varphi(fX\,\overline{\rhd}_V gY) &=& \tilde\varphi(f\rho_V(X)(g)Y+fg(X\rhd_V Y))\\
&=& f\rho_V(X)(g)\psi(Y)+ fg\psi(X\rhd_V Y)\\
&=& f\rho(\psi(X))(g)\psi(Y)+ fg(\psi(X)\rhd \psi(Y))\\
&=& f\psi(X) \rhd g\psi(Y)=\ \tilde\varphi(fX)\rhd\tilde\varphi(gY),
\end{eqnarray*}
which implies that $\tilde\varphi$ is a post-Lie algebra homomorphism. Moreover, we have
$$
\rho(\tilde\varphi(fX)) = \rho(f\psi(X))\ =\ f\rho(\psi(X))= f\rho_V(X) = \overline{\rho}_V(fX),
$$
which implies that $\tilde\varphi:{\rm PostLR}(V)\to L$ is a post-Lie-Rinehart algebra homomorphism. Meanwhile, a post-Lie-Rinehart algebra homomorphism from ${\rm PostLR}(V)$ to $L$ is determined by its restriction on $V$, so the desired homomorphism $\tilde\varphi$ is unique.
\end{proof}

\begin{exam}\label{ex:free-post-Lie}
The free post-Lie algebra on one generator is realized as ${\rm PostLie}(\bk\huaO)$, where $\bk\huaO$ is the magma algebra of planar rooted trees with the left grafting operation $\curvearrowright$~\cite{ML}.

According to Theorem~\ref{thm:free-post-LR}, given any linear map $f_{\bk\huaO}:\bk\huaO\to \Der(\lR)$, the action post-Lie-Rinehart algebra
${\rm PostLR}(\bk\huaO)=(\lR,\lR\otimes {\rm Lie}(\bk\huaO),[-,-]_{\lR\otimes {\rm Lie}(\bk\huaO)},\overline{\rho}_{\bk\huaO},\overline{\rhd}_{\bk\huaO})$
is the free post-Lie-Rinehart algebra generated by $(\bk\huaO,\curvearrowright,f_{\bk\huaO})$.
\end{exam}

We mention that the free object in the category of pre-Lie-Rinehart algebras with trace has been constructed via aromatic non-planar trees in \cite{FMM}, and such algebraic structures can also be characterised by universal geometric properties \cite{MKV,LMKequiv}.

\subsection{The universal enveloping algebra of a post-Lie-Rinehart algebra}

Given a post-Lie-Rinehart algebra $\mathscr{\lR}=(\lR,\lL,[-,-]_\lL,\rho,\rhd)$, since $\lL$ is a Lie $\lR$-algebra,
we have a Hopf $\lR$-algebra $(\mathscr U_\lR(\lL),\Delta,\varepsilon,S,\iota_\lR)$ by Proposition~\ref{prop:uea-LRA}. Now we provide our main theorem in this section, which generalizes \cite[Theorem 2.8]{LST}.

\begin{theorem}\label{thm:uea-post-LR}
For a post-Lie-Rinehart algebra $\mathscr{\lR}=(\lR,\lL,[-,-]_\lL,\rho,\rhd)$,
the Hopf $\lR$-algebra $(\mathscr U_\lR(\lL),\Delta,\varepsilon,S,\iota_\lR)$
has a $\bk$-linear product $\rhd$ recursively defined by
\begin{eqnarray}
\label{eq:exten-1}
f\rhd X &=& fX,\\
\label{eq:exten-2}
x\rhd f &=& x(f),\\
\label{eq:exten-3}
x\rhd yX &=& (x\rhd y)X + y(x\rhd X),\\
\label{eq:exten-4}
xX \rhd Y &=& x\rhd (X\rhd Y) -(x\rhd X)\rhd Y,
\end{eqnarray}
for any $f,g\in \lR$, $x,y\in \lL$ and $X,Y\in \mathscr U_\lR(\lL)$.
Then the tuple $(\mathscr U_\lR(\lL),\Delta,\varepsilon,S,\iota_\lR,\rhd)$ is a weak post-Hopf algebroid over $\lR$,
which is called the {\bf universal enveloping algebra} of $\mathscr{\lR}$.

\end{theorem}

\begin{proof}

According to the Oudom-Guin construction in \cite{ELM}, the post-Lie product $\rhd$ can be extended to the universal enveloping algebra $\mathscr U(\lL)$ of the Lie algebra $\lL$ via Eqs.~\eqref{eq:exten-1}--\eqref{eq:exten-2} with $f=1_\lR$, and also \eqref{eq:exten-3} and \eqref{eq:exten-4} with $X,Y\in \mathscr U(\lL)$. As the first part,
we show that such a construction can be generalized for the post-Lie-Rinehart algebra $\mathscr{\lR}$, namely the post-Lie product $\rhd$ of $\lL$ can be extended further to the $\lR$-algebra $\mathscr U_\lR(\lL)$ via formulas ~\eqref{eq:exten-1}--\eqref{eq:exten-4}.

Let $\mathscr U(\lL)=T(\lL)/(xy-xy-[x,y]_\lL\,|\,x,y\in\lL)=\bk\oplus \bar{\mathscr U}(\lL)$, where $\bar{\mathscr U}(\lL)=\lL\mathscr U(\lL)$. Correspondingly, we have a direct sum of $\bk$-linear subspaces
$\mathscr U_\lR(\lL)=\lR \oplus \bar{\mathscr U}(\lL)/\mathscr J$, where $\mathscr J$ is the ideal of $\mathscr U(\lL)$ given by
$$\mathscr J=((fx)y-x(fy)\in \bar{\mathscr U}(\lL)\,|\,f\in\lR,\,x,y\in\lL).$$

First note that the action $\rhd$ of $\lR$ on $\mathscr U_\lR(\lL)$ via Eq.~\eqref{eq:exten-1} is well-defined, since $\mathscr U_\lR(\lL)$ is an $R$-module by the left multiplication of $\lR$. So we still need to verify that $\mathscr U_\lR(\lL)$ acts on $\lR$ properly by Eqs.~\eqref{eq:exten-2} and \eqref{eq:exten-4}, and then the linear map
$$\rhd:\bar{\mathscr U}(\lL)/\mathscr J\otimes \bar{\mathscr U}(\lL)/\mathscr J\to \bar{\mathscr U}(\lL)/\mathscr J$$
is also well-defined by Eqs.~\eqref{eq:exten-3} and \eqref{eq:exten-4}. In order to do so,
we just check that
\begin{equation}\label{eq:uea_rhd_lR}
X(xy-yx-[x,y]_\lL)Y\rhd f=0,\quad \forall f\in \lR,\ x,y\in\lL,\ X,Y\in T(\lL),
\end{equation}
and then
$\bar{\mathscr U}(\lL) \rhd\mathscr J\subseteq \mathscr J$, $\mathscr J\rhd \lR = \{0\}$ and $\mathscr J \rhd \bar{\mathscr U}(\lL)\subseteq \mathscr J$ successively.

For the proof of Eq.~\eqref{eq:uea_rhd_lR}, we apply Eqs.~\eqref{Post-L-1}, \eqref{eq:exten-3} and \eqref{eq:exten-4} to see that
\begin{eqnarray*}
&&wX(xy-yx-[x,y]_\lL)Y\rhd f\\
&=&  w\rhd(X(xy-yx-[x,y]_\lL)Y\rhd f)
- (w\rhd X)(xy-yx-[x,y]_\lL)Y\rhd f\\
&&- X((w\rhd x)y-y(w\rhd x)-[w\rhd x,y]_\lL)Y\rhd f
 - X(x(w\rhd y)-(w\rhd y)x-[x,w\rhd y]_\lL)Y\rhd f\\
&& - X(xy-yx-[x,y]_\lL)(w\rhd Y)\rhd f
\end{eqnarray*}
for any $f\in \lR$, $w,x,y\in\lL$ and $X,Y\in T(\lL)$.
So by induction on the degree of polynomials it reduces to check that
$$(xy-yx-[x,y]_\lL)X\rhd f=0,\quad \forall f\in \lR,\ x,y\in\lL,\ X\in T(\lL).$$
By Eqs.~\eqref{eq:exten-2}--\eqref{eq:exten-4}, we find that
\begin{eqnarray*}
(xy-yx-[x,y]_\lL)X\rhd f
&=&x(y(X\rhd f)) - y(x(X\rhd f))- [x,y]_\rhd(X\rhd f)\\
&&\quad -\big(x\rhd(y\rhd X)-y\rhd(x\rhd X)-[x,y]_\rhd\rhd X\big)\rhd f,
\end{eqnarray*}
which vanishes by Eqs.~\eqref{Post-L-2}, \eqref{eq:exten-3} and the anchor $\rho:\lL_\rhd\to\Der(\lR)$ as a Lie homomorphism. Hence, Eq.~\eqref{eq:uea_rhd_lR} holds.

Now it remains to check that $\bar{\mathscr U}(\lL) \rhd\mathscr J\subseteq \mathscr J$, and then $\mathscr J\rhd \lR = \{0\}$ and $\mathscr J \rhd \bar{\mathscr U}(\lL)\subseteq \mathscr J$.
First one can use Eqs.~\eqref{eq:exten-3} and \eqref{eq:exten-4} repeatedly to see that $\bar{\mathscr U}(\lL)\rhd\mathscr J\subseteq \mathscr J$, since the following initial step holds:
\begin{eqnarray*}
&& w\rhd((fx)y-x(fy))\\
&\stackrel{\eqref{eq:exten-3}}{=}& (w\rhd fx)y +(fx)(w\rhd y)- (w\rhd x)(fy)-x(w\rhd fy)\\
&\stackrel{\eqref{eq:PLR-2}}{=}& (w(f)x+f(w\rhd x))y +(fx)(w\rhd y)- (w\rhd x)(fy) -x(w(f)y+ f(w\rhd y))\\
&=& (w(f)x)y-x(w(f)y) + (f(w\rhd x))y- (w\rhd x)(fy)+ (fx)(w\rhd y) -x(f(w\rhd y))\ \in \ \mathscr J
\end{eqnarray*}
for any $f\in R$ and $w,x,y\in L$.

Next we check that  $\mathscr J\rhd \lR = \{0\}$ and $\mathscr J \rhd \bar{\mathscr U}(\lL)\subseteq \mathscr J$ simultaneously by induction on the degree of polynomials in $\mathscr J$.
First fix $J\in \mathscr J$ and assume that $J'\rhd \lR=\{0\}$ (resp.  $J'\rhd \bar{\mathscr U}(\lL)\subseteq \mathscr J$) for any polynomial $J'\in \mathscr J$ of degree not greater than $J$. Then for any $x\in L$ and $X\in \lR$ (resp. $X\in \bar{\mathscr U}(\lL)$),
\begin{eqnarray*}
xJ\rhd X &\stackrel{\eqref{eq:exten-4}}{=}&
x\rhd (J\rhd X) -(x\rhd J)\rhd X.
\end{eqnarray*}
As  shown previously $\bar{\mathscr U}(\lL)\rhd\mathscr J\subseteq \mathscr J$, we have $x\rhd (J\rhd X)$, $(x\rhd J)\rhd X=0$
(resp. $x\rhd (J\rhd X), (x\rhd J)\rhd X\in \mathscr J$) by the assumption, so $xJ\rhd X=0$ (resp. $xJ\rhd X\in \mathscr J$).
As a result, it is enough to show that $((fx)y-x(fy))\mathscr U(\lL)\rhd \lR =\{0\}$ and $((fx)y-x(fy))\mathscr U(\lL)\rhd \bar{\mathscr U}(\lL)\subseteq \mathscr J$.

For any $f,g\in \lR$, $x,y\in L$, we have
\begin{eqnarray*}
&&((fx)y-x(fy))\rhd g\\
&\stackrel{\eqref{eq:exten-4}}{=}& fx \rhd(y\rhd g) - (fx\rhd y)\rhd g -x\rhd(fy\rhd g) + (x\rhd fy)\rhd g\\
&\stackrel{\eqref{eq:PLR-1},\eqref{eq:PLR-2}}{=}& fx(y(g)) - f(x\rhd y)(g) - x(fy(g)) + x(f)y(g) + f(x\rhd y)(g)\ \\
&=&\ 0,
\end{eqnarray*}
so $((fx)y-x(fy))\rhd \lR=\{0\}$. On the other hand, if $w\in L$, then
\begin{eqnarray*}
&&((fx)y-x(fy))\rhd w\\
&\stackrel{\eqref{eq:exten-4}}{=}& fx\rhd(y\rhd w) - (fx\rhd y)\rhd w -x\rhd(fy\rhd w) + (x\rhd fy)\rhd w\\
&\stackrel{\eqref{eq:PLR-1},\eqref{eq:PLR-2}}{=}& fx\rhd(y\rhd w) - f(x\rhd y)\rhd w -x\rhd (fy\rhd w)  + (x(f)y+f(x\rhd y))\rhd w\\
&\stackrel{\eqref{eq:PLR-1}}{=}& f(x\rhd(y\rhd w)) - x\rhd f(y\rhd w) + x(f)(y\rhd w)\ \\
&\stackrel{\eqref{eq:PLR-2}}{=}&\ 0.
\end{eqnarray*}
Then assuming that $((fx)y-x(fy))\rhd X\in \mathscr J$ for $X\in \bar{\mathscr U}(\lL)$,
we see that
\begin{eqnarray*}
((fx)y-x(fy))\rhd wX
&\stackrel{\eqref{eq:exten-3},\eqref{eq:exten-4}}{=}&
(((fx)y-x(fy))\rhd w)X + (fx\rhd w)(y\rhd X) \\
 && \quad + (y\rhd w)(fx\rhd X) - (x\rhd w)(fy\rhd X)\\
 && \quad - (fy\rhd w)(x\rhd X) + w(((fx)y-x(fy))\rhd X)\\
&\stackrel{\eqref{eq:PLR-1}}{=}&  (f(x\rhd w))(y\rhd X)  - (x\rhd w)(fy\rhd X)\\
 && \quad + (y\rhd w)(fx\rhd X) - (f(y\rhd w))(x\rhd X)\\
&& \qquad + w(((fx)y-x(fy))\rhd X) \in \mathscr J,
\end{eqnarray*}
as $(fx)Xy-xX(fy)\in \mathscr J$ for any $x,y\in L$ and $X\in \bar{\mathscr U}(\lL)$ by the definition of $\mathscr J$. So we also have $((fx)y-x(fy))\rhd \bar{\mathscr U}(\lL)\subseteq \mathscr J$.

Then for any $Y\in \lR$ (resp. $Y\in\bar{\mathscr U}(\lL)$), we obtain that
\begin{eqnarray*}
((fx)y-x(fy))X\rhd Y &\stackrel{\eqref{eq:exten-4}}{=}&
(\underline{(fx)y-x(fy)})\rhd (X\rhd Y) -  (\underline{((fx)y-x(fy))\rhd X})\rhd Y\\
&& \quad + (\underline{x(fy\rhd X)- (fx)(y\rhd X)})\rhd Y \\
&& \quad  + (\underline{(fy)(x\rhd X) - y(fx\rhd X)})\rhd Y = 0\ (\mbox{resp. }\in \mathscr J)
\end{eqnarray*}
by induction on the degree of polynomials in $\mathscr J$, since the above elements underlined all belong to $\mathscr J$.
Hence, the proof that $\mathscr J\rhd \lR = \{0\}$ and $\mathscr J \rhd \bar{\mathscr U}(\lL)\subseteq \mathscr J$ has completed.

\medskip
For the second part, we show that the Hopf $\lR$-algebra $(\mathscr U_\lR(\lL),\Delta,\varepsilon,S,\iota_\lR)$ with the $\bk$-linear product $\rhd$  satisfies
conditions~\eqref{eq:post-HAD-1}--\eqref{eq:post-HAD-4},
so $(\mathscr U_\lR(\lL),\iota_\lR,\rhd)$ is a weak post-Hopf algebroid.

Recall that
$\mathscr U_\lR(\lL)=\lR \oplus \bar{\mathscr U}(\lL)/\mathscr J$,
where $\mathscr U(\lL)=\bk\oplus\bar{\mathscr U}(\lL)$ is the universal enveloping algebra of the post-Lie algebra $(\lL,\rhd)$.
First we see that conditions~\eqref{eq:post-HAD-1'}--\eqref{eq:post-HAD-2'} clearly hold for $\mathscr U_\lR(\lL)$ by definitions of the product $\rhd$ and the counit $\varepsilon$ of $\mathscr U_\lR(\lL)$. For instance, Eqs.~\eqref{eq:exten-1}, \eqref{eq:exten-2} and \eqref{eq:exten-4} tell us that
\begin{eqnarray*}
X\rhd 1_\lR &=& \begin{cases}
 0, & X\in \bar{\mathscr U}(\lL)/\mathscr J,\\
 X, & X\in \lR,
\end{cases}
\end{eqnarray*}
so Eq.~\eqref{eq:post-HAD-1'} holds.
By the recursion formulas \eqref{eq:exten-3} and \eqref{eq:exten-4}, the compatibility condition~\eqref{eq:PLR-1} for the post-Lie-Rinehart algebra $\mathscr{\lR}$ is extended to condition~\eqref{eq:post-HAD-2} for $\mathscr U_\lR(\lL)$.

In order to verify conditions~\eqref{eq:post-HAD-1}, \eqref{eq:post-HAD-3} and \eqref{eq:post-HAD-4} for $\mathscr U_\lR(\lL)$, we only need to check the situations involving elements in the base algebra $\lR$, since they already hold for the original post-Hopf algebra $\mathscr U(\lL)$. For any $f,g\in \lR$, $X,Y\in \mathscr U_\lR(\lL)$,
\begin{eqnarray*}
\Delta(f\rhd X) &=& \Delta(fX)\ =\ fX_1\otimes_\lR X_2=(f\rhd X_1)\otimes_\lR (1_\lR\rhd X_2),\\
\Delta(X\rhd f) &=& (X\rhd f)\otimes_\lR 1_\lR\ =\ (\varepsilon(X_1)X_2\rhd f)\otimes_\lR 1_\lR \ \stackrel{\eqref{eq:post-HAD-1'},\,\eqref{eq:post-HAD-2}}{=} \ (X_1\rhd f)\otimes_\lR (X_2\rhd 1_\lR),
\end{eqnarray*}
so Eq.~\eqref{eq:post-HAD-1} holds. By induction on the degree of polynomials in $\mathscr U_\lR(\lL)$, we check that
\begin{eqnarray*}
f\rhd XY &\stackrel{\eqref{eq:exten-1}}{=}& fXY \ \stackrel{\eqref{eq:exten-1}}{=}\ (f\rhd X)(1_\lR\rhd Y),\\
x\rhd fY &\stackrel{\eqref{eq:exten-3},\,\eqref{eq:PLR-2}}{=}& f(x\rhd Y)+x(f)Y \ \stackrel{\eqref{eq:exten-2}}{=}\  (x\rhd f)Y + f(x\rhd Y),\\
xX\rhd fY &\stackrel{\eqref{eq:exten-4}}{=}& x\rhd (X\rhd fY) -(x\rhd X)\rhd fY\\
&\stackrel{\eqref{eq:post-HAD-1}}{=}& x\rhd (X_1\rhd f)(X_2\rhd Y)\\
&&- ((x\rhd X_1)\rhd f)(X_2\rhd Y) - (X_1\rhd f)((x\rhd X_2)\rhd Y)\\
&=& (x\rhd (X_1\rhd f))(X_2\rhd Y) - ((x\rhd X_1)\rhd f)(X_2\rhd Y)\\
&& + (X_1\rhd f)(x\rhd (X_2\rhd Y)) - (X_1\rhd f)((x\rhd X_2)\rhd Y)\\
&\stackrel{\eqref{eq:exten-4}}{=}& (xX_1\rhd f)(X_2\rhd Y)+(X_1\rhd f)(xX_2\rhd Y),
\end{eqnarray*}
so Eq.~\eqref{eq:post-HAD-3} holds. Also, we have
\begin{eqnarray*}
f\rhd(X\rhd Y) &\stackrel{\eqref{eq:exten-1}}{=}& f(X\rhd Y) \ \stackrel{\eqref{eq:post-HAD-2}}{=} \ fX \rhd Y \ \stackrel{\eqref{eq:exten-1}}{=}\ f(1_\lR\rhd X)\rhd Y,\\
X\rhd(f\rhd Y) &\stackrel{\eqref{eq:exten-1}}{=}& X\rhd fY \ \stackrel{\eqref{eq:post-HAD-3}}{=}\ (X_1\rhd f)(X_2\rhd Y)\ \stackrel{\eqref{eq:post-HAD-2}}{=}\ (X_1\rhd f)X_2\rhd Y\ =\ X_1(X_2\rhd f)\rhd Y.
\end{eqnarray*}
Hence, Eq.~\eqref{eq:post-HAD-4} also holds.
\end{proof}

\begin{remark}
  When $\lR=\bk$, this result recovers \cite[Theorem 2.8]{LST}, which shows that  for a  post-Lie algebra $(\frakg,[-,-]_\frakg,\rhd)$, its universal enveloping algebra $(\uu(\g),\rhd)$ is a post-Hopf algebra.
\end{remark}

Now we consider  an action post-Lie-Rinehart algebra $(\lR,\lL=\lR\otimes \frakg,[-,-]_\lL,\overline{\rho},\overline{\rhd})$ given in Proposition~\ref{thm:post-LR-derivation} from a post-Lie algebra $(\frakg,[-,-]_\frakg,\rhd)$ with an action $\rho:\frakg_\rhd\to \Der(\lR)$. The action $\rho$ can be lifted to a module algebra action of $\mathscr U(\frakg_\rhd)$ on $\lR$ by the universal property of $\uu(\frakg_\rhd)$.
On the other hand, for the post-Hopf algebra $(\mathscr U(\frakg),\rhd)$, there is a natural Hopf algebra isomorphism $\Phi:\mathscr U(\frakg_\rhd)\to\mathscr U(\frakg)_\rhd$, called  the Oudom-Guin isomorphism~\cite{ELM}.
Pulling back by $\Phi^{-1}$, we obtain a module algebra action $\rightharpoonup$ of $\mathscr U(\frakg)_\rhd$ on $\lR$.
Then by Theorem~\ref{thm:action-post-Hopf}, we have an action post-Hopf algebroid $(\lR\otimes \mathscr U(\frakg),\Delta,\varepsilon,S,\iota,\overline{\rhd},\theta)$ over $\lR$.

\begin{coro}\label{coro:universal-action}
With the above notations, the universal enveloping algebra $\mathscr U_\lR(\lL)$ of the action post-Lie-Rinehart algebra $(\lR,\lL=\lR\otimes \frakg,[-,-]_\lL,\overline{\rho},\overline{\rhd})$ is a post-Hopf algebroid, which is isomorphic to the action post-Hopf algebroid $(\lR\otimes \mathscr U(\frakg),\Delta,\varepsilon,S,\iota,\overline{\rhd},\theta)$ over $\lR$.

\end{coro}
\begin{proof}
First by Proposition~\ref{prop:uea-LRA}, $(\mathscr U_\lR(\lL),\Delta,\varepsilon,S)$ is a Hopf $\lR$-algebra, since $\lL=\lR\otimes \frakg$ is a Lie $\lR$-algebra.
Also by Theorem~\ref{thm:uea-post-LR}, we know that $(\mathscr U_\lR(\lL),\Delta,\varepsilon,S,\iota,\overline{\rhd})$ is a weak post-Hopf algebroid.

Moreover, we have a natural identification between
$\mathscr U_\lR(\lL)$ and the tensor product Hopf $\lR$-algebra $\lR\otimes \mathscr U(\frakg)$.
Elements in $\mathscr U_\lR(\lL)$ have the form $\sum_i f_ix_i$ with $f_i\in \lR$ and $x_i\in\mathscr U(\frakg)$, and the post-Hopf product $\overline{\rhd}$ of $\mathscr U_\lR(\lL)$ extended from the post-Lie product $\overline{\rhd}$ of $\lL$ is just given by Eqs.~\eqref{action-post-Hopf-1}. On the other hand, an algebra anti-automorphism $\theta$ of $\mathscr U_\lR(\lL)$ can be given by \eqref{action-post-Hopf-2}, so that $\mathscr U_\lR(\lL)$ is the action post-Hopf algebroid $(\lR\otimes \mathscr U(\frakg),\Delta,\varepsilon,S,\iota,\overline{\rhd},\theta)$ over $\lR$ as in Theorem~\ref{thm:action-post-Hopf}.
\end{proof}

Given a magma algebra $(V,\curvearrowright)$, there exists a unique extended product $\rhd$ on the tensor algebra $T(V)$ such that
the coshuffle Hopf algebra $(T(V),\cdot,\Delta^{\rm cosh},S,\rhd_V)$ is a post-Hopf algebra~\cite[Theorem~2.9]{LST}.
Now for any linear map $f_V:V\to \Der(\lR)$, it can be extended to a Lie algebra homomorphism from $\Lie(V)$ to $\Der(\lR)$. Since  $T(V)$ is the universal enveloping algebra of the Lie algebra $\Lie(V)$, it follows that $\lR$ is a $T(V)$-module algebra. Then using the Hopf algebra isomorphism $K:T(V)_{\rhd_V}\to T(V)$, which is called Gavrilov's $K$-map (\cite[Theorem~3]{AEMM}
), we obtain the following $T(V)_{\rhd_V}$-module algebra action $\rightharpoonup$ on $\lR$ recursively given by:
\begin{eqnarray*}
v\rightharpoonup f &=& f_V(v)(f),\\
vv_1\cdots v_n\rightharpoonup f &=& v\rightharpoonup(v_1\cdots v_n\rightharpoonup f) - (v\rhd_V v_1\cdots v_n)\rightharpoonup f
\end{eqnarray*}
for any $f\in \lR$ and $v,v_1,\dots,v_n\in V$.
Then by Theorem~\ref{thm:action-post-Hopf}, we obtain
the  action post-Hopf algebroid over $\lR$,
\begin{eqnarray*}
T_\lR(V) &\coloneqq& (\lR \otimes T(V),\iota,\overline{\rhd}_V,\theta).
\end{eqnarray*}

By Corollary \ref{coro:universal-action}, we obtain the following characterization of the universal enveloping algebra of the free post-Lie-Rinehart algebra.

\begin{coro}\label{coro:free-post-Hopf}
The universal enveloping algebra of the free post-Lie-Rinehart algebra   generated by a magma algebra $(V,\curvearrowright)$ together with a linear map $f_V:V\to\Der(\lR)$ given in Theorem \ref{thm:free-post-LR} is   the action post-Hopf algebroid $T_\lR(V)=(\lR \otimes T(V),\iota,\overline{\rhd}_V,\theta)$ over $\lR$.
\end{coro}

\begin{exam}\label{ex:free-post-Hopf}
As shown in Example~\ref{ex:free-post-Lie}, the free post-Lie-Rinehart algebra  generated by the triple $(\bk\huaO,\curvearrowright,f_{\bk\huaO})$ is the action post-Lie-Rinehart algebra
$${\rm PostLR}(\bk\huaO)=(\lR,\lR\otimes {\rm Lie}(\bk\huaO),[-,-]_{\lR\otimes {\rm Lie}(\bk\huaO)},\overline{\rho}_{\bk\huaO},\overline{\rhd}_{\bk\huaO}).$$

According to Corollary~\ref{coro:free-post-Hopf},
the universal enveloping algebra of ${\rm PostLR}(\bk\huaO)$ is the action post-Hopf algebroid $T_\lR(\bk\huaO)=(\lR \otimes T(\bk\huaO),\iota,\overline{\rhd}_{\bk\huaO},\theta)$ over $\lR$, where the post-Hopf algebra $(T(\bk\huaO),\rhd_{\bk\huaO})$ is the universal enveloping algebra of the free post-Lie algebra ${\rm PostLie}(\bk\huaO)$, and
its sub-adjacent Hopf algebra $T(\bk\huaO)_{\rhd_{\bk\huaO}}$ is the Grossman-Larson Hopf algebra of ordered planar forests.

\end{exam}

\section{Post-Lie-Rinehart algebras and post-Hopf algebroids for numerical analysis}\label{sec:app}

In this section, we detail the natural post-Lie-Rinehart algebras and post-Hopf algebroid structures appearing in numerical analysis.

Let $\MM$ be a smooth manifold equipped with a frame basis $(E_d)$, defined globally for clarity.
The smooth vector fields $\XM$ are naturally equipped with the Jacobi bracket $\llbracket -,-\rrbracket_J$ and the (curvature-free) Weitzenb\"ock connection,
\[
Y \rhd X=\sum_i Y[x^i] E_i,\quad X=x^i E_i,
\]
where we use the Einstein summation convention.
Let us further assume that the connection $\rhd$ has constant torsion $\nabla T=0$, or equivalently that $\MM$ is locally a Lie group.
In this case, the torsion defines the bracket:
\[
[X,Y]=-T(X,Y)=x^i y^j \llbracket E_i,E_j\rrbracket_J,\quad X=x^i E_i,\quad Y=y^j E_j.
\]
The two brackets are linked by the identity
\[
\llbracket X,Y \rrbracket_J=[X,Y]+X\rhd Y- Y\rhd X.
\]
Then, the connection algebra $(\XM,[-,-],\rhd)$ is a post-Lie algebra \cite{ML}.

As a specific application of Theorem~\ref{thm:uea-post-LR} and Theorem~\ref{thm:GL-bialgebroid}, we have the following result.
\begin{prop}
The tuple $(\CM,\XM,[-,-],\rho,\rhd)$ is a post-Lie-Rinehart algebra and its universal enveloping algebra $(\mathscr U_{\CM}(\XM),\cdot,\Delta^{\rm cosh},\rhd)$ is a weak post-Hopf algebroid over $\CM$, where the anchor map $\rho$ is given by
\[
\rho(X)(\phi)=X[\phi],\quad X=x^i E_i,\quad \phi\in \CM.
\]
Moreover, $(\mathscr U_{\CM}(\XM),*_\rhd,\Delta^{\rm cosh})$ is the associated Grossman-Larson bialgebroid.
\end{prop}

We rely on specific sub-structures of the universal enveloping algebra $\mathscr U_{\CM}(\XM)$ in numerical analysis.
Let $F\in \XM$ be a given smooth Lipschitz vector field. Then, $F$ generates a post-Lie sub-algebra $(\XMF,[-,-],\rhd)$:
\[
\XMF=\Span_{\RR}(
F,\ F\rhd F,\ [F,F\rhd F],\ (F\rhd F)\rhd F,\ F\rhd (F\rhd F),\ \dots
).
\]
This yields the Hopf subalgebras $(\mathscr U(\XMF),\cdot,\Delta^{\rm cosh})$ and $(\mathscr U(\XMF),*_\rhd,\Delta^{\rm cosh})$, commonly represented with ordered planar forests \cite{LMK}.
Analogously, let $\End_F(\XM)$ be the space of endomorphisms on $\XM$ generated from $F$, $\rhd$ and $[-,-]$:
\[
\End_F(\XM)=\Span_{\RR}(
\id,\ F\rhd -,\ -\rhd F,\ [-,F\rhd F],\ [F,-\rhd F],\ [F,F\rhd -],\ (-\rhd F)\rhd F,\ \dots
).
\]
Let $\CMF$ be the subalgebra of $\CM$ generated by the following functions, with elements called {\bf aromas} in numerical analysis \cite{CM, IQT, Bo}:
\[
u_1^{i_1}(E_{i_2})\cdots u_{n-1}^{i_{n-1}}(E_{i_n}) u_n^{i_n}(E_{i_1}),\quad u_k\in \End_F(\XM),\quad u_k(X)=u_k^i(X)E_i.
\]

Following Corollaries \ref{coro:universal-action} and \ref{coro:GL_Hopf_algebroid}, we obtain the following action Hopf algebroid structures using the anchor map $\rho$.
\begin{theorem}
The tuple $(\CMF,\CMF\otimes\XMF,[-,-],\rho,\rhd)$ is an action post-Lie-Rinehart algebra, with elements called {\bf aromatic} vector fields.
Its universal enveloping algebra is the action post-Hopf algebroid $(\CMF\otimes \mathscr U(\XMF),\cdot,\Delta^{\rm cosh},\rhd,\theta)$ over $\CMF$, with the associated Grossman-Larson Hopf algebroid $(\CMF\otimes\mathscr U(\XMF),*_\rhd,\Delta^{\rm cosh},\theta)$ over $\CMF$.
\end{theorem}

\begin{remark}
Let $\Div\colon \XM\rightarrow\CM$ be the divergence operator associated to the connection $\rhd$:
\[
  \Div(X)=E_i[x^i], \quad X=x^i E_i\in \XM.
\]
The divergence restricts to aromatic vector fields $\Div\colon \XMF\rightarrow\CMF$, but aromas are not in general all generated by the divergence: $\Div(\XMF)\subsetneq \CMF$, as observed in \cite{FMM}.
\end{remark}

Let the flow $\varphi_t(y_0)$, $t\in\mathbb R$ be the solution of the initial value problem on $\MM$:
\begin{equation}
\label{eq:ODE}
y'(t)=F(y(t)),\quad y(0)=y_0\in\MM.
\end{equation}
The Taylor expansion of the pullback by the exact flow of \eqref{eq:ODE} is described by the Grossman-Larson exponential:
\[
\varphi_t^*\phi=\exp^{*_\rhd}(tF)\rhd\phi,\quad \exp^{*_\rhd}(X)=\sum_{n=0}^\infty \frac{1}{n!} X^{*_\rhd n},\quad \phi\in \mathcal{C}^\infty(\MM).
\]
Most existing numerical flows $\psi_t$ for solving \eqref{eq:ODE} have Taylor expansions that rewrite with formal series over the Hopf algebra $\mathscr U(\mathfrak{X}_{tF}(\MM))$, that we denote by $\overline{\mathscr U(\mathfrak{X}_{tF}(\MM))}$.
This is the case for Lie-group methods \cite{IMNZ}, and, in particular, for the Lie-Euler method, that follows the geodesics for $\rhd$:
\begin{equation}
\label{eq:Lie_Euler}
\psi_t(y_0)=\exp_{y_0}(tF(y_0)).
\end{equation}
The Taylor expansion of the pullback by Lie-group methods is typically expressed using the concatenation product. For instance, the Taylor expansion of the Lie-Euler flow \eqref{eq:Lie_Euler} is described by the concatenation exponential
\[
\psi_t^*\phi=\exp^\cdot(tF)\rhd\phi,\quad \exp^\cdot(X)=\sum_{n=0}^\infty \frac{1}{n!} X^{\cdot n},\quad \phi\in \mathcal{C}^\infty(\MM),
\]
where the concatenation product acts on functions as the following differential operator
\[
(X_n\cdots X_1)[\phi]=x_n^{i_n}\dots x_1^{i_1} E_{i_n}[\dots E_{i_1}[\phi]\dots],\quad \phi\in \mathcal{C}^\infty(\MM).
\]
A method $\psi_t$ is of order $p$ if the Taylor expansions of $\varphi_t^*\phi$ and $\psi_t^*\phi$ match for all terms of order at most $p$. The Lie-Euler method \eqref{eq:Lie_Euler} is of order one only for instance.

\begin{remark}
The Grossman-Larson product $*_\rhd$ corresponds to the composition of flows, in the sense that
\[
(\psi^2_t\circ \psi^1_t)^*\phi=(S^2_t *_\rhd S^1_t)\rhd \phi, \quad \psi^{k*}_t\phi=S^k_t\rhd \phi, \quad S^k_t\in \overline{\mathscr U(\mathfrak{X}_{tF}(\MM))}.
\]
The description of the composition of flows in simpler cases has been extensively studied in the numerical literature using tree structures and Butcher series for the description of the Grossman-Larson Hopf algebra.
In particular, the dual of the Grossman-Larson Hopf algebra $T(\bk\huaO)_{\rhd_{\bk\huaO}}$
is isomorphic to the Munthe-Kaas-Wright Hopf algebra \cite{MW, EFR}.
In the pre-Lie context, the dual of the action pre-Hopf algebroid of aromatic forests \cite{CM,IQT} is isomorphic to the Butcher-Connes-Kreimer Hopf algebroid \cite{Bo, BL}.
\end{remark}

Let us now define aromatic numerical flows by relying on the post-Hopf algebroid structure of $\CMF\otimes \mathscr U(\XMF)$ (see also \cite{MKV, Bo}). Note that the differential operators defined using functions in $\CMF$ do still act on all smooth functions in $\mathcal{C}^\infty(\MM)$.
\begin{defn}
A flow $\psi_t$ for solving \eqref{eq:ODE} is {\bf aromatic} if the Taylor expansion of the pullback $\psi_t^*\phi$ is a formal series in $\overline{\mathcal{C}^\infty_{tF}(\MM)\otimes\mathscr U(\mathfrak{X}_{tF}(\MM))}$:
\[\psi_t^*\phi=S_t\rhd \phi,\quad S_t\in \overline{\mathcal{C}^\infty_{tF}(\MM)\otimes\mathscr U(\mathfrak{X}_{tF}(\MM))},\quad \phi\in \mathcal{C}^\infty(\MM).\]
\end{defn}

\begin{exam}
\label{exam:aromaticLieEuler}
Aromatic methods are straightforwardly obtained from Lie-group methods by substituting $F=f^i E_i$ with an aromatic vector field $t\hat{F}_t\in \mathcal{C}^\infty_{tF}(\MM)\otimes \mathfrak{X}_{tF}(\MM)$.
For instance, apply the Lie-Euler method \eqref{eq:Lie_Euler} with the preprocessed aromatic vector field
\[
t\hat{F}_t=
tF
+\frac{t^2}{2}F\rhd F
-\frac{t^3}{3} (F\rhd F)\rhd F
-\frac{t^3}{12} E_i[F[f^i]] F
+\frac{t^3}{6} [F,F\rhd F].
\]
Then, the resulting flow $\psi_t$ is aromatic.
\end{exam}

The application of aromatic flows in geometric numerical integration is better understood through backward error analysis \cite{HLW,LMK,Rahm}.
Under technical assumptions satisfied by large classes of numerical methods, a numerical flow may be understood as the exact flow of a formal modified differential equation $y'=\tilde{F}_t(y)$ driven by a modified vector field $t\tilde{F}_t$, that is,
\[
\psi_t^*\phi=\exp^{*_\rhd}(t\tilde{F}_t)\rhd\phi,\quad t\tilde{F}_t\in \overline{\mathcal{C}^\infty_{tF}(\MM)\otimes\mathscr U(\mathfrak{X}_{tF}(\MM))},\quad \phi\in\mathcal{C}^\infty(\MM).
\]
The geometric properties of the numerical method $\psi_t$ are read directly on $\tilde{F}_t$.
A case of importance in the literature lies in divergence-free features $\Div(F)=\Div(\tilde{F}_t)=0$ and is linked with volume-preserving flows.
In this context, one is interested in the creation of consistent divergence-free numerical methods. The design of such methods is an open problem of numerical analysis, even in the pre-Lie context.
It is showed in \cite{IQT, DL} that non-exact numerical flows described by $\overline{\mathscr U(\mathfrak{X}_{tF}(\MM))}$ cannot satisfy $\Div(\tilde{F}_t)=0$, so that the use of aromatic numerical flows is crucial to the design of divergence-free methods.

\begin{appn}
While the creation of exact divergence-free methods is an open problem, one can create pseudo-divergence-free methods by using the post-Hopf algebroid $\CMF\otimes \mathscr U(\XMF)$.
In particular, a calculation yields that the aromatic method from Example \ref{exam:aromaticLieEuler} satisfies (at least formally) third order of preservation: $\Div(\tilde{F}_t)=O(t^3)$.
In contrast, the original Lie-Euler method \eqref{eq:Lie_Euler} satisfies only the first order estimate $\Div(\tilde{F}_t)=O(t)$.
\end{appn}

Given an aromatic flow, the computation of the associated modified vector field $\tilde{F}_t$ is not direct and relies on the so-called substitution law. In the pre-Lie case, it is described with non-planar aromatic trees in \cite{Bo, BL} and the divergence-free modified aromatic vector fields $\tilde{F}_t$ are fully characterised in \cite{LMMKV, La, DL}.
The extension of these results to the post-Lie context and the design and study of convenient algebraic structures for representing aromatic vector fields and aromas are key to the creation of divergence-free aromatic numerical methods and are matter for future work.

\bibliographystyle{amsplain}

\end{document}